\documentclass[12pt,reqno]{amsart}
\usepackage{amssymb}
\usepackage{amsthm}
\usepackage{amsmath}
\usepackage{fullpage}
\usepackage{epsfig}
\usepackage{hyperref,mathrsfs,verbatim,amscd,latexsym}
\usepackage{amssymb}
\usepackage{url}

\usepackage{mathrsfs}
\usepackage{dsfont}

\renewcommand{\setminus}{\smallsetminus}

\renewcommand{\le}{\leqslant}
\renewcommand{\ge}{\geqslant}
\renewcommand{\leq}{\leqslant}
\renewcommand{\geq}{\geqslant}
\renewcommand{\epsilon}{\varepsilon}

\makeatletter

\numberwithin{equation}{section} \numberwithin{figure}{section}

\theoremstyle{plain}
\newtheorem{thm}{Theorem}[section]
\newtheorem{lem}[thm]{Lemma}

\newtheorem*{lem*}{Lemma}
\newtheorem{prop}[thm]{Proposition} 
\newtheorem*{prop*}{Proposition}
\newtheorem{lem-defn}[thm]{Lemma-Definition}
\newtheorem{prop-defn}[thm]{Proposition-Definition}
\newtheorem{cor}[thm]{Corollary}
\newtheorem*{cor*}{Corollary}

\theoremstyle{definition}
\newtheorem{defn}[thm]{Definition}
\newtheorem*{defn*}{Definition}
\newtheorem{example}[thm]{Example}

\theoremstyle{remark}
\newtheorem{rem}[thm]{Remark}
\newtheorem*{rem*}{Remark}

\newtheorem{claim*}{Claim}

\newcommand{\lemref}[1]{Lemma \ref{lem:#1}}

\providecommand{\abs}[1]{\left\vert#1\right\vert}
\providecommand{\norm}[1]{\left\Vert#1\right\Vert}
\providecommand{\size}[1]{\abs{#1}}

\def\G{\Gamma}
\def\g{\gamma}
\def\GX{{\G\bs X}}
\def\Ga{\G_\alpha}
\def\GL{\mathrm{GL}}
\def\Sp{\mathrm{Sp}}

\def\AGG{\mathcal{A}(G,\G)}
\def\AGS{\mathcal{A}(G,S)}
\def\AGj{\mathcal{A}(G,S^j)}

\providecommand{\Cay}[1]{\mathrm{Cay}(#1;S)}
\newcommand{\CayG}{\Cay{\G}}
\newcommand{\CayGa}{\Cay{\Ga}}

\DeclareMathOperator{\Isom}{Isom}           
\DeclareMathOperator{\sgn}{sgn}             

\def\eqdef{\overset{\text{def}}{=}}
\def\bs{\backslash}
\def\restrict{\!\upharpoonright}            

\def\Pmod{\Lambda}
\def\Pmodp{\Pmod^{(p)}}
\def\Pmodt{\Pmod^{(2)}}

\def\dpY{d^p_Y}
\def\dY{d_Y}

\def\MY{\mathcal{M}(Y)}                     
\def\MX{\mathcal{M}(X)}
\def\WX{\mathcal{W}(X)}                     
\def\WGX{\mathcal{W}^{\G}(X)}               
\def\MV{\mathcal{M}(V)}
\def\WV{\mathcal{W}(V)}

\def\clin{\mathrm{c}_\mathrm{lin}}           
\def\cp{\mathrm{c}_p}                        
\def\ct{\mathrm{c}_2}                        
\def\ci{\mathrm{c}_\infty}                   

\def\uY{c_Y}
\def\uYp{\uY^p}

\newcommand{\Dvar}[2]{\abs{\nabla_{#1}(#2)}_p}
\newcommand{\Dmp}[1]{\Dvar{\mu}{#1}}

\newcommand{\Ep}{\mathcal{E}^{(p)}}
\newcommand{\Emp}{\Ep_\mu}
\newcommand{\Emnp}{\Ep_{\mu^n}}
\newcommand{\Emmp}{\Ep_{\mu^m}}
\newcommand{\Emmt}{\mathcal{E}^{(2)}_{\mu^m}}

\newcommand{\Ap}{A^{(p)}}
\newcommand{\Amp}{\Ap_\mu}
\newcommand{\Amnp}{\Ap_{\mu^n}}

\newcommand{\vE}{\vec{E}}
\newcommand{\ve}{\vec{e}}
\newcommand{\vp}{\vec{p}}
\newcommand{\vpp}{\tilde{p}}
\newcommand{\avp}{\abs{\vp}}
\newcommand{\avpp}{\abs{\vpp}}

\newcommand\muX{\mu_X}                       
\newcommand\muXj{\mu_X^j}                    
\newcommand\muGa{\mu_{G,\alpha}}             
\newcommand\muqGa{\muGa^q}                   
\newcommand\mubGX{\bar\mu_{G,X}}             

\newcommand\bn{\bar\nu}

\newcommand\BC{$(\mathrm{B}\mathcal{C})$}
\newcommand\NC{$(\mathrm{N}\mathcal{C})$}
\newcommand\FC{$(\mathrm{F}\mathcal{C})$}

\newcommand\FS{$(\mathrm{F}\mathcal{S})$}

\def\ie{i.e.\ }

\def\wrt{w.r.t.\ }

\def\hf{\frac{1}{2}}
\def\fop{\frac{1}{p}}
\def\foq{\frac{1}{q}}
\def\fqp{\frac{q}{p}}

\makeatother

\DeclareSymbolFont{symbolsC}{U}{txsyc}{m}{n}
\SetSymbolFont{symbolsC}{bold}{U}{txsyc}{bx}{n}
\DeclareFontSubstitution{U}{txsyc}{m}{n}
\DeclareMathSymbol{\coloneqq}{\mathrel}{symbolsC}{66}

\newcommand\remove[1]{}

\newcommand{\rnote}[1]{}

\newcommand{\Lip}{\mathrm{Lip}}

\newcommand{\N}{\mathbb{N}}
\newcommand{\Z}{\mathbb{Z}}
\newcommand{\R}{\mathbb{R}}

\newcommand{\B}{\mathscr{B}}

\newcommand{\Ae}{\mathrm{ae}}
\newcommand{\Prob}{\mathscr{P}}

\newcommand{\e}{\varepsilon}

\newcommand{\E}{\mathbb{E}}

\DeclareMathOperator{\diam}{diam}

\theoremstyle{plain}
\newtheorem{theorem}{Theorem}[section]
\newtheorem{corollary}[theorem]{Corollary}

\begin{document}


\title{Poincar\'e inequalities, embeddings, and wild groups}
\author{ Assaf Naor}\thanks{A.N. is supported in part by
NSF grants CCF-0635078 and CCF-0832795, BSF grant 2006009, and the
Packard Foundation.}
\address{Courant Institute of Mathematical Sciences, New York University, 251 Mercer Street, New York NY 10012, USA.}
\email{naor@cims.nyu.edu}
\author{Lior Silberman}\thanks{This work was
completed when L.S. was an intern at Microsoft Research, Redmond WA,
June--August 2004. He wishes to thank Microsoft Research for their
hospitality.}
\address{Department of Mathematics, University of British Columbia, 1984 Mathematics Road, Vancouver  BC   V6T 1Z2, Canada.}
\email{lior@math.ubc.ca}

\subjclass[2010]{20F65,58E40,46B85}
\keywords{Gromov's random groups, fixed points, Poincar\'e inequalities}

\begin{abstract}
We present geometric conditions on a metric space $(Y,d_Y)$ ensuring that almost surely, any isometric action on $Y$ by Gromov's expander-based random group has a common fixed point. These geometric conditions involve uniform convexity and the validity of nonlinear Poincar\'e inequalities, and they are stable under natural operations such as scaling, Gromov-Hausdorff limits, and Cartesian products. We use methods from metric embedding theory to establish the validity of these conditions for a variety of classes of metric spaces, thus establishing new fixed point results for actions of Gromov's ``wild groups".
\end{abstract}

\date{}
\maketitle

\section{Introduction}\label{sec:intro}

We establish the existence of finitely generated groups with strong
fixed point properties. The seminal work on this topic is Gromov's
construction~\cite{Gromov:RandGp} of random groups from expander
graph families, leading to a solution~\cite[Sec.\
7]{HigsonLafforgueSkandalis:BC-CounterEx} of the Baum-Connes
conjecture for groups, with coefficients in commutative
$C^*$-algebras. Here we study Gromov's construction, highlighting
the role of the geometry of the metric space on which the group
acts. As a result, we isolate key properties of the space acted upon
that imply that any isometric action of an appropriate random group
has a common fixed point. Using techniques from the theory of metric
embeddings in order to establish these properties, we obtain new
fixed point results for a variety of spaces that will be described
below. This answers in particular a question of
Pansu~\cite{Pansu:EinGedi} (citing Gromov).  In fact, we prove the
stronger statement that for every
Euclidean building $B$ (see~\cite{KL97}), there exists a torsion-free
hyperbolic group for which every isometric action on $\ell_2(B)$ has a common fixed point (this statement extends to appropriate $\ell_2$ products of more than one building). Thus, following Ollivier's
terminology~\cite{Ollivier:RandGpSurvey}, Gromov's groups are even
``wilder" than previously shown.

For $p\ge 1$ say that a geodesic metric space $(Y,d_Y)$ is
$p$-uniformly convex if there exists a constant $c>0$ such that for
every $x,y,z \in Y$, every geodesic segment $\gamma:[0,1]\to  Y$
with $\gamma(0)=y$, $\gamma(1)=z$, and every $t\in[0,1]$ we have:
\begin{equation}\label{eq:conv intro}
\dY(x,\gamma(t))^p \leq (1-t)\dY(x,y)^p + t\dY(x,z)^p - c
t(1-t)\dY(y,z)^p.
\end{equation}
It is immediate to check that~\eqref{eq:conv intro} can hold only
for $p\ge 2$. The inequality~\eqref{eq:conv intro} is an obvious
extension of the classical notion of $p$-uniform convexity of Banach
spaces (see, e.g., \cite{BCL94}), and when $p=2$ it is an extension
of the $\mathrm{CAT}(0)$ property (see, e.g.,
\cite{BridsonHaelfliger:NonPosSpc}).

We shall say that a  metric space $(Y,d_Y)$ admits a sequence of
\emph{high girth $p$-expanders} if there exists $k\in \mathbb N$,
$\gamma,\eta>0$, and a sequence of $k$-regular finite graphs
$\{G_n=(V_n,E_n)\}_{n=1}^\infty$ with $\lim_{n\to
\infty}|V_n|=\infty$ such that the length of the shortest non-trivial closed
path (the ``girth'') in $G_n$ is at least $\eta\log |V_n|$, and such that
for every $f:V_n\to Y$ we have,
\begin{equation}\label{eq:def poin intro}
\frac{1}{|V_n|^2} \sum_{u,v\in V_n}
d_Y(f(u),f(v))^p\le\frac{\gamma}{|E_n|}\sum_{uv\in E_n}
d_Y(f(u),f(v))^p.
\end{equation}
When $Y=\R$ it is well-known that inequality \eqref{eq:def poin
intro} with $p=2$ is equivalent to the usual notion of combinatorial
expansion (for a survey on expander graphs see
\cite{HooryLinialWigderson:ExpanderSurvey}, especially Section 2).
It is less well-known~\cite{Mat97} that this is true for all
$1<p<\infty$; we reproduce the proof in Lemma \ref{lem:Matousek}. It
is also worth noting that unless $Y$ consists of a single point, the
sequence of graphs considered must necessarily be a sequence of
combinatorial expanders.

As we shall see later, a large class of metric spaces of
interest consists of spaces that
are both $p$-uniformly convex and admit a sequence of high girth
$p$-expanders. In fact, in all cases that we study, the Poincar\'e
inequality~\eqref{eq:def poin intro} holds for {\em every} sequence
of combinatorial expanders. It is an open problem whether the
existence of a sequence of bounded degree graphs
satisfying~\eqref{eq:def poin intro} implies the same conclusion for
all combinatorial expanders, but we will not deal with this issue
here as the existence statement suffices for our purposes.

Gromov's remarkable construction~\cite{Gromov:RandGp} of random
groups is described in detail in Section~\ref{sec:graphmodel}. In
order to state our results, we briefly recall it here. Given a
(possibly infinite) graph $G=(V,E)$, and integers $j,d\in \mathbb
N$, a probability  distribution over groups $\Gamma$ associated to
$G$ and generated by $d$ elements $s_1,\ldots,s_d$ is defined as
follows. Orient the edges of $G$ arbitrarily. For every edge $e\in
E$ choose a word $w_e$ of length $j$ in $s_1,\ldots,s_d$  and their
inverses uniformly at random from all such $(2d)^j$ words, such that
the random variables $\{w_e\}_{e\in E}$ are independent. Each cycle
in $G$ induces a random relation obtained by traversing the cycle,
and for each edge $e$ of the cycle, multiplying by either $w_e$ or
$w_e^{-1}$, depending on whether $e$ is traversed according to its
orientation or not. These relations induce the random group
$\Gamma=\Gamma(G,d,j)$.

Our main result is:
\begin{thm}\label{thm:main intro}
Assume that a geodesic metric space $(Y,d_Y)$ is $p$-uniformly
convex and admits a sequence of high girth $p$-expanders
$\{G_n=(V_n,E_n)\}_{n=1}^\infty$. Then for all $d\ge 2$ and $j\ge 1$
 with probability tending to
$1$ as $n\to \infty$, any isometric action of the group
$\Gamma(G_n,d,j)$ on $Y$ has a common fixed point.
\end{thm}
It was shown in~\cite{Gromov:RandGp,Ollivier:GraphRandGp,ArzhantsevaDelzant:ExampleRandGps_preprint}
that for every $d\ge 2$, for large enough $j$ (depending only on $d$ and the
parameters $k,\eta$), the group $\Gamma(G_n,d,j)$ is torsion-free
and hyperbolic with probability tending to $1$ as $n\to \infty$.

 Using a variety of results
and techniques from the theory of metric embeddings, we present a
list of metric spaces $(Y,d_Y)$ for which the conditions of
Theorem~\ref{thm:main intro} are satisfied\footnote{Note that our
conditions on the metric space $(Y,d_Y)$ in Theorem~\ref{thm:main
intro} are closed under  $\ell_p$ sums $(\oplus_{s=1}^N Y_s)_p$,
provided that in~\eqref{eq:def poin intro}, the same high-girth
expander sequence works for all the $Y_s$. This holds true in all
the examples that we present, for which~\eqref{eq:def poin intro} is
valid for every connected graph, with $\gamma$ depending only on
$p$, the spectral gap of the graph, and certain intrinsic geometric
parameters of $Y$.}. These spaces include all Lebesgue spaces
$L_q(\mu)$ for $1<q<\infty$, and more generally all Banach lattices
which are $p$-uniformly convex for some $p\in [2,\infty)$. Moreover,
they include all (possibly infinite dimensional) Hadamard manifolds
(in which case $p=2, c=1$), all Euclidean buildings ($p=2, c=1$,
again), and all $p$-uniformly convex Gromov hyperbolic  metric
spaces of bounded local geometry. It was asked by Pansu
in~\cite{Pansu:EinGedi} whether for every symmetric space or
Euclidean building an appropriate random group has the fixed point
property. Our results imply that this is indeed the case. As a
corollary, by a ``gluing" construction of~\cite{ABJLMS09} (see
also~\cite[Sec.\ 3.3]{FisherSilberman:Noaction}) it follows that
there exists a torsion-free group that has the fixed point property
with respect to {\em all} the spaces above. This yields one
construction of ``wild groups". Alternatively, one could follow the
original approach of Gromov~\cite{Gromov:RandGp}, who considers the
group $\Gamma=\Gamma(G,d,j)$, where the graph $G$ is the disjoint
union of an appropriate subsequence of the expanders
$\{G_n\}_{n=1}^\infty$ from Theorem~\ref{thm:main intro}, which is a
torsion-free group with positive
probability~\cite{Gromov:RandGp,Ollivier:GraphRandGp,ArzhantsevaDelzant:ExampleRandGps_preprint}. For this
group $\Gamma$, Pansu asked~\cite{Pansu:EinGedi} whether it has the
fixed point property with respect to {\em all} symmetric spaces and
all buildings of type $\tilde A_n$. Our result implies that for
every $d,j$, almost surely $\Gamma$ will indeed have this fixed
point property, and also on all $\ell_2$ products of such spaces.

\begin{thm}\label{thm:main also} Let $G$ be the disjoint union of
a family of high-girth combinatorial expanders (that is, of a family
of graphs for which a single $\gamma$ applies in~\eqref{eq:def poin
intro} for all $\R$-valued functions).  Let $d\geq 2$ and $j\geq 1$.
Then with probability $1$ the group $\Gamma(G,d,j)$ has the
fixed-point property for isometric actions on all $p$-uniformly
convex Banach Lattices, all buildings associated to linear groups,
all non-positively curved symmetric spaces, and all $p$-uniformly
convex Gromov hyperbolic spaces.  The fixed-point property also
holds for an $\ell_p$-product of $p$-uniformly convex spaces, as
long as the constant in \eqref{eq:def poin intro} is uniformly
bounded for these spaces.
\end{thm}

Problems similar to those studied here were also investigated
in~\cite{IzekiNayatani:HarmonicMaps,IzekiNayatni:FP_RandGp}, where
criteria were introduced that imply fixed point properties of random
groups in \.Zuk's triangular model~\cite{Zuk:ShortRelsPropT}.
These criteria include a Poincar\'e-type inequality similar
to~\eqref{eq:def poin intro}, with the additional requirement that the
constant $\gamma$ is small enough (in our normalization, they require
$p=2$ and $\gamma<2$).  Unfortunately, it is not known whether it is
possible to establish such a strong Poincar\'e inequality for
the spaces studied here, except for $\mathrm{CAT}(0)$ manifolds, trees,
and a specific example of an $\tilde A_2$ building
(see~\cite{IzekiNayatani:HarmonicMaps,IzekiNayatni:FP_RandGp}).
Our approach is insensitive to the exact value of $\gamma$
in~\eqref{eq:def poin intro}.  In fact, $\gamma$ can be allowed to grow to
infinity with $|V_n|$; see~\eqref{eq:1/4} and Theorem~\ref{thm:techthm} below.

It was shown in~\cite{Pansu:Sp_n1_FLp_preprint} that any cocompact
lattice $\G$ in $\Sp_{n,1}(\R)$ admits a fixed-point-free action by
linear isometries on $L_p$ for any $p\geq 4n+1$. Also, $\Gamma$ acts
by isometries on the symmetric space of $\Sp_{n,1}(\R)$ (which is a
Hadamard manifold) without fixed points. Thus, while it is
known~\cite{Gromov:RandGp,Silberman:RandGpT} that Gromov's random
groups have property $(T)$, our results do not from follow from
property $(T)$ alone.
See~\cite{FisherMargulis:PropTRigid,BaderFurmanGelanderMonod:LpPropT}
for a discussion of the relation between property $(T)$ and fixed
points of actions on $L_p$.

We end this introduction by noting that the above gluing-type
construction based on Theorem~\ref{thm:main intro} yields a
non-hyperbolic group. This is necessary, since it was shown
in~\cite{Yu:HypGpProperActLp} that any hyperbolic group admits a
proper (and hence fixed-point free) isometric action on an
$L_p(\mu)$ space for $p$ large enough. It remains open  whether
there exists a \emph{hyperbolic} group with the fixed-point property
on all symmetric spaces and Euclidean buildings. Such a group would
have no infinite linear images. (This is related to the well known
problem of the existence of a hyperbolic group which is not
residually finite.)

\medskip

\noindent{\bf Overview of the structure of this paper.} In
Section~\ref{sec:fixed-pts} we recall some background on fixed point
properties of groups, and how they are classically proved. The
natural approach to finding a fixed point from a bounded orbit by
considering the average (or center of mass) of the orbit requires
appropriate definitions in general uniformly convex metric spaces;
this is discussed in Section~\ref{sec:averaging}. But, in our
situation orbits are not known to be bounded, so the strategy is to
average over certain bounded subsets of an orbit. The hope is that
by iterating this averaging procedure we will converge to a fixed
point. It turns out that this approach works in the presence of
sufficiently good Poincar\'e inequalities; this is explained in
Section~\ref{sec:mainthm}, a key technical tool being
Theorem~\ref{thm:poincare-to-avg} (before reading
Section~\ref{sec:mainthm}, readers should acquaint themselves with
the notations and definitions of Section~\ref{sec:graphmodel}, which
recalls Gromov's construction of random groups). We prove the
desired Poincar\'e inequalities (in appropriate metric spaces) via a
variety of techniques from the theory of metric embeddings;
Section~\ref{sec:poincare} and Section~\ref{sec:nagata} are devoted
to this topic.

\medskip
\noindent{\bf Asymptotic notation.} We use $A \lesssim B$, $B
\gtrsim A$ to denote the estimate $A \leq CB$ for some absolute
constant $C$; if we need $C$ to depend on parameters, we indicate
this by subscripts, thus $A \lesssim_p B$ means that $A \leq C_p B$
for some $C_p$ depending only on $p$. We shall also use the notation
$A\asymp B$ for $A\lesssim B\ \wedge\ B\lesssim A$.

\section{Background on fixed point properties of groups}\label{sec:fixed-pts}
We start by setting some terminology.
\begin{defn} Let $\G$ be a finitely generated group, let $(Y,d_Y)$ be a metric space, and let
$\rho\colon\G\to\Isom(Y)$ be an action by isometries. We say that
the action satisfies the condition:
\begin{list}{}{}
\item[(N),] if the image $\rho(\G)$ is finite;
\item[(F),] if the image $\rho(\G)$ has a common fixed point;
\item[(B),] if some (equiv.\ every) $\G$-orbit in $Y$ is bounded.
\end{list}
For a class $\mathcal{C}$ of  metric spaces, we say that $\G$ has
property \NC, \FC\ or \BC\ if every action
$\rho\colon\Gamma\to\Isom(Y)$, where $Y\in\mathcal{C}$, satisfies
the respective condition.
\end{defn}
The Guichardet-Delorme Theorem
\cite{Guichardet:FP_2_T,Delorme:T_2_FP} asserts that if $H$ is
Hilbert space then $\G$ has property (FH) if and only if it has
Kazhdan property $(T)$. The reader can take this as the definition
of property $(T)$ for the purpose of this paper.

Fixed-point properties can have algebraic implications for the
group's structure.  For example, finitely generated linear groups
have isomorphic embeddings into linear groups over local fields, and
these latter groups act by isometries on non-positively curved
spaces with well-understood point stabilizers. For completeness and
later reference, we include the following simple lemma.

\begin{lem}[Strong non-linearity]\label{lem:nonlinear}
Let $\mathcal{S}$ be the class of the symmetric spaces and buildings
associated to the groups $\GL_n(F)$, where $F$ is a non-Archimedean
local field. Let $\G$ be a finitely generated group with property
\FS. Then any homomorphic image of $\G$ into a linear group is
finite.
\end{lem}
\begin{proof}
Let $\G$ be finitely generated group with property \FS. Let $K$ be a
field, and let $\rho\colon\G\to\GL_n(K)$ be a homomorphism. Without
loss of generality we can assume $K$ to be the field generated by
the matrix elements of the images of the generators of $\G$, and
then let $A\subset K$ be the set of matrix elements of the images of
all elements of $\Gamma$. Clearly $\rho(\G)$ is finite iff $A$ is a
finite set, and \cite[Lem.\ 2.1]{BreuillardGelander:TopTitsAlt}
reduces the finiteness of $A$ to showing that the image of $A$ under
any embedding of $K$ in a  local field $F$ is relatively compact.
Hence, let $\iota\colon K\to F$ be such an embedding.  This induces
a group homomorphism $\GL_n(K)\to\GL_n(F)$ which we also denote
$\iota$. Composing with $\rho$ we obtain a homomorphism
$\iota\circ\rho\colon\G\to\GL_n(F)$.  Now let $S$ be the symmetric
space (if $F$ is Archimedean) or Bruhat-Tits building (if $F$ is
non-Archimedean) associated to $\GL_n(F)$.  Since $\GL_n(F)$ is a
group of isometries of $S$, the image of $\iota\circ\rho$ must fall
in the stabilizer in $\GL_n(F)$ of a point of $S$.  Since these
stabilizers are compact subgroups of $\GL_n(F)$ we are done.
\end{proof}

Lemma~\ref{lem:nonlinear} implies, via our results as stated in the
introduction, that Gromov's wild groups are not isomorphic to linear
groups. Alternatively, this fact also follows from the result
of~\cite{GuentnerHigsonWeinberger:LinearNovikov} that asserts that
any linear group admits a coarse embedding into Hilbert space, while
it was shown in~\cite{Gromov:RandGp} that Gromov's random group does
not admit such an embedding (indeed, this was the original
motivation for Gromov's construction). It also follows from
Lemma~\ref{lem:nonlinear} that all linear homomorphic images of
Gromov's random group are finite. In fact, it was later observed
in~\cite{FisherSilberman:Noaction} that the random group has no
finite images, and hence also no linear images, since finitely
generated linear groups are residually finite.

It is clear that the condition (B) is implied by either condition
(N) or (F).  When $Y$ is complete and $p$-uniformly convex the
converse holds as well (weaker notions of uniform convexity suffice
here). We recall the standard proof of this fact below, since it
illustrates a ``baby version" of the averaging procedure on
uniformly convex spaces that will be used extensively in what
follows.
\begin{lem}[``Bruhat's Lemma'']\label{lem:bruhat}
Let $Y$ be a uniformly convex geodesic metric space. Then the
condition (B) for isometric actions on $Y$ implies condition (F).
\end{lem}
\begin{proof}
To any bounded set $A\subset Y$ associate its \emph{radius function}
$r_A(y) = \sup_{a\in A} d_Y(y,a)$.  For any $a\in A$ and $y_0,
y_1\in Y$ let $y_{1/2}$ be a midpoint of the geodesic segment
connecting them.  By equation~\eqref{eq:conv intro} we have that
$\dY(a, y_{1/2})^p$ is less than the average of $\dY(a,y_0)^p$ and
$\dY(a,y_0)^p$  by a positive quantity depending only on
$d_Y(y_0,y_1)$ and growing with it.  It follows that the diameter of
the set $C_\epsilon\subset Y$ on which $r_A$ exceeds its minimum by
no more than $\epsilon$ goes to zero with $\epsilon$. Since $Y$ is
complete it follows that $r_A(y)$ has a unique global minimizer,
denoted $\ci(A)\in Y$, and called the $\emph{circumcenter}$ of $A$.
Since its definition involved only the metric on $Y$, the
circumcenter map is equivariant under isometries of $Y$.  It follows
that the circumcenter of a bounded orbit for a group action is a
fixed point.
\end{proof}



\section{Averaging on metric spaces}\label{sec:averaging}

We saw above how to find a fixed point from a bounded orbit, by
forming a kind of ``average" (circumcenter) along the orbit.  When
the orbits are not known to be bounded, it is not possible to form
such averages. However, if $\Gamma$ (generated by $S=S^{-1}$) acts
on a $p$-uniformly convex space $Y$, it is possible to average over
small pieces of the orbit: passing from a point $y$ to an
appropriately defined average of the finite set $\{ sy \}_{s\in S}$
(the precise notion of averaging is described below). Under suitable
conditions this averaging procedure is a contraction on $Y$, leading
to a fixed point. In practice we will need to average over small
balls rather than just $S$ itself, but the idea remains the same.

``Averaging'' means specifying a function that associates to Borel
probability measures $\sigma$ on $Y$ a point $c(\sigma) \in Y$, in a
well-behaved manner. We will not axiomatize the needed properties,
instead defining the procedures we will use.  We start with a
particularly simple example. In what follows all measures are
assumed to have finite support---this suffices for our purposes, and
the obvious generalizations are standard.

\begin{example}
Let $Y$ be a Banach space, and let $\sigma$ be a (finitely
supported) probability measure on $Y$. The vector-valued integral \[
\clin(\sigma) = \int_Y y d\sigma(y) \]  is called the \emph{linear}
center of mass of $\sigma$.
\end{example}

This center of mass behaves well under linear maps, but its metric
properties are not so clear.  Thus even  for the purpose of proving
fixed-point properties for actions on  $L_p$ we use a nonlinear
averaging method, related to a metric definition of linear averaging
on Hilbert space. This is a standard method in metric geometry (see
for example~\cite[Chapter\ 3]{Jost:NonPosCurv}).

For a metric space $(Y,d_Y)$ we write $\MY$ for the space of
probability measures on $Y$ with finite support.  Generally it is
enough to assume below that the measures have finite $p$th moment
for the appropriate $p \geq 2$ but we will not use such measures
since our groups are finitely generated.


\subsection{Uniformly convex metric spaces and the geometric center of mass}
\label{sub:uniform-convexity}

We continue with our complete metric space $(Y,d_Y)$. A
\emph{geodesic segment} in $Y$ is an isometry $\gamma\colon I\to Y$
where $I\subseteq \R$ is a closed interval, and the metric on $I$ is
induced from the standard metric on $\R$.  If the endpoints $a<b$ of
$I$ are mapped to $y,z\in Y$ respectively, we will say that the
segment $\gamma$ \emph{connects} $y$ to $z$, and usually denote it
by $[y,z]$.  Moreover, for any $t\in[0,1]$ we will use $[y,z]_t$ to
denote $\gamma((1-t)a + tb)$.  This notation obscures the fact that
there may be distinct geodesic segments connecting $y$ to $z$, but
this will not be the case for the spaces we consider (see below).

We now assume that $Y$ is a \emph{geodesic} metric space, i.e., that
every two points of $Y$ are connected by a geodesic segment.
\begin{defn}
Let $2 \leq p < \infty$. $Y$ is said to be $p$-\emph{uniformly
convex} if there exists a constant $\uY>0$ such that for every
$x,y,z \in Y$, every geodesic segment $[y,z]\subseteq  Y$, and every
$t\in[0,1]$ we have:
\begin{equation}\label{eq:pconvex}
\dY(x,[y,z]_t)^p \leq (1-t)\dY(x,y)^p + t\dY(x,z)^p - \uYp
t(1-t)\dY(y,z)^p.
\end{equation}
We say that $Y$ is \emph{uniformly convex} if it is $p$-uniformly
convex for some $p\geq 2$.
\end{defn}
The above definition is an obvious extension of the notion of
$p$-uniform convexity of Banach spaces (see, e.g.,
\cite{Figiel76,BCL94}). For concreteness, an $L_p(\mu)$ space is $p$
uniformly convex if $p\in [2,\infty)$ and $2$-uniformly convex if
$p\in (1,2]$.  In Hilbert space specifically, \eqref{eq:pconvex}
with $p=2$ and $\uY=1$ is an equality, and it follows that the same
holds for conclusions such~\eqref{eq:dsigma-growth} below.
We also note that it is easy to see that a uniformly convex metric
space is uniquely geodesic by examining midpoints.


We now recall the notion of $\mathrm{CAT}(0)$ spaces. For $y_1,y_2,y_3\in Y$,
choose $Y_1,Y_2,Y_3 \in \R^2$ such that $\norm{Y_i-Y_j}_2 =
d_Y(y_i,y_j)$ for any $i,j$.   Such a triplet of \emph{reference
points} always exists, and is unique up to a global isometry of
$\R^2$.  It determines a triangle $\Delta = I_{12}\cup I_{23}\cup
I_{13}$ consisting of three segments of lengths $d_Y(y_i,y_j)$.  Any
choice of three geodesic segments $\gamma_{ij}\colon I_{ij}\to Y$
connecting $y_i,y_j$ gives a \emph{reference map} $R\colon \Delta
\to Y$. We say that $(Y,d_Y)$ is a $\mathrm{CAT}(0)$ space if for every three
points $y_i\in Y$ every associated reference map $R$ does not
increase distances. It is a standard fact
(see~\cite{BridsonHaelfliger:NonPosSpc}) that $(Y,d_Y)$ is a $\mathrm{CAT}(0)$
space iff it is $2$-uniformly convex with the constant $\uY$
in~\eqref{eq:pconvex} equal to $1$. $\mathrm{CAT}(0)$ spaces are
$p$-uniformly convex for all $p \in [2,\infty)$ since the plane
$\R^2$ is $p$-uniformly convex (it is isometric to a subset of
$L_p$).

Assume that $(Y,d_Y)$ is $p$-uniformly convex. Let $\sigma \in \MY$.
Integrating equation~\eqref{eq:pconvex}  we see that for all $y,z\in
Y$:
\begin{equation}\label{eq:convex-cm}
\uYp t(1-t)\dY(y,z)^p \leq (1-t)d_p(\sigma,y)^p + td_p(\sigma,z)^p -
d_p(\sigma,[y,z]_t)^p,
\end{equation}
where for $w\in Y$ we write
$$
d_p(\sigma,w)=\left(\int_{Y} d_Y(u,w)^pd\sigma(u)\right)^{1/p}.
$$
Now let $d = \inf_{y\in Y} d_p(\sigma,y)$, and assume
$d_p(\sigma,y),d_p(\sigma,z) \leq (d^p+\epsilon)^{1/p}$. Letting
$w\in Y$ denote the midpoint of any geodesic segment connecting $y$
and $z$ we have $d_p(\sigma,w) \geq d$ and hence:
$$ \frac{\uYp}{4}\dY(y,z)^p \leq d^p + \epsilon - d^p = \epsilon.$$
In other words, the set of $y\in Y$ such that $d_p(\sigma,y)^p$ is
at most $d^p + \epsilon$ has diameter $\lesssim_{c_Y} \epsilon
^{1/p}$. By the completeness of $Y$, there exists a unique point
$\cp(\sigma)\in Y$ such that $d_p(\sigma,\cp(\sigma)) = d$.

To justify the notation $\ci(A)$ introduced in Lemma \ref{lem:bruhat}
notes that $d_\infty(\sigma,y) = r_A(y)$ where $A$ is the essential
support of $\sigma$.

\begin{defn} The point $\cp(\sigma)$ is called the
\emph{geometric} center of mass of $\sigma$.  We will also use the
term $p$-center of mass when we wish to emphasize the choice of
exponent. The point $\ci(A)$ is called the \emph{circumcenter} of
$A$.
\end{defn}

\begin{rem} Consider the special case of the real line
with the standard metric, and of $\sigma = t\delta_1 +
(1-t)\delta_0$. Then $\cp(\sigma)$ represents a weighted average
of $0,1\in\R$. We note that (except for special values of $t$),
the $\cp(\sigma)$ vary depending on $p$.
\end{rem}

We now apply equation (\ref{eq:convex-cm}) where $z =
\cp(\sigma)$. Still using $d_p(\sigma,[y,z]_t) \geq d$ we get:
$$ \uYp t(1-t)\dY(\cp(\sigma),y)^p \leq (1-t)(d_p(\sigma,y)^p - d^p).$$
Dividing by $1-t$ and letting $t\to 1$ we get the following useful
inequality:
\begin{equation}\label{eq:dsigma-growth}
d_p (\sigma,y)^p \geq d_p (\sigma, \cp(\sigma))^p +
\uYp\dY(\cp(\sigma),y)^p.
\end{equation}

\subsection{Random walks}\label{sub:random-walks}

Let $X$ be a discrete set. Following Gromov~\cite{Gromov:RandGp} we
shall use the following terminology.

\begin{defn}
By a \emph{random walk} (or a Markov chain) on $X$ we shall mean a function
$\mu\colon X\to\MX$. The space of random walks will be denoted $\WX$.
\end{defn}

For a random walk $\mu$ and $x\in X$ we will denote below the
measure $\mu(x)$  by either $\mu_x$ or $\mu(x\to\cdot)$. The latter
notation emphasizes the view of $\mu$ as specifying the transition
probabilities of a Markov chain on $X$. For  $\nu\in \MX$,
$\mu,\mu'\in\WX$ we write
$$
\nu*\mu \eqdef \int_X d\nu(x) \mu_x  \in \MX.
$$
The map $x\mapsto(\mu'*\mu)_x \eqdef \mu'_x * \mu$
  defines a random walk on $X$. For $n\in \N$ we define
  inductively $\mu^{n+1}\eqdef \mu^n*\mu$.

Let $\nu$ be a measure on $X$.  We say that a random walk
$\mu\in\WX$ is $\nu$-\emph{reversible}, if we have
\begin{equation}\label{eq:nu-symm}
d\nu(x)d\mu(x\to x') = d\nu(x')d\mu(x'\to x),
\end{equation}
as an equality of measures on $X\times X$. If $X$ is finite, we can
assume that $\nu$ is a probability measure. In general integrating
equation~\eqref{eq:nu-symm} \wrt $x'$, we see that $\nu$ is a
\emph{stationary measure} for $\mu$, in the sense that $\nu*\mu =
\nu$. 

Finally, let the discrete group $\G$ act freely on $X$. The induced
action of $\G$ on $\MX$ preserves $\MX$ in this case. The space of
$\G$-equivariant random walks will be denoted $\WGX$.  Moreover, we
have a quotient space $\GX$.  Fixing a probability measure $\bn$ on
$\GX$, we will call $\mu\in\WGX$ $\bn$-reversible if it is
$\nu$-reversible where $\nu$ is the pull-back of $\bn$ defined by
$\int_X fd\nu = \int_{\Gamma\backslash X}\left(\sum_{\g\in\G} f(\g
x)\right)d\bar\nu(x)$ for any $f\in C_\mathrm{c}(X)$.

\subsection{Averaging of equivariant functions into uniformly convex spaces}
\label{sub:averaging} Continuing with the notation used so far, let
$\mu\in\WGX$ be reversible \wrt the probability measure $\bn$ on
$\GX$. Let $(Y,d_Y)$ be a $p$-uniformly convex metric space on which
$\G$ acts by isometries.

Now let $f\colon X \to Y$ be  $\G$-equivariant. For $x\in X$, the
push-forward $f_*\mu_x$ is a probability measure on $Y$ with finite
support (the image of the support of $\mu_x$ under $f$). We set:
\begin{equation}\label{eq:def gradient}
\Dmp{f}(x) = \left( \int_X d\mu(x\to x')
\dY\left(f(x),f(x')^p\right) \right)^{1/p},
\end{equation}
\begin{equation}\label{eq:def energy}
 \Emp(f) = \hf \int_\GX \left(\Dmp{f}(x)\right)^pd\bn(x),
 \end{equation}
\begin{equation}\label{eq:def finite energy}
 B(X,Y) = \left\{ f\in C(X,Y)^\G \mid \Emp(f)<\infty \right\}.
 \end{equation}

For $f,g \in C(X,Y)^\G$, the function $x\mapsto d_Y(f(x),g(x))$ is
$\G$-invariant, and we can hence set
$$d_p(f,g) \eqdef \left( \int_\GX \dY(f(x),g(x))^p d\bn(x) \right)^{1/p}.$$
This  defines a (possibly infinite) complete metric. The triangle
inequality gives:

\begin{lem}\label{lem:bound-energies}We have,
\begin{enumerate}
\item Let $f,g\in C(X,Y)^\G$.  Assume $d_p(f,g)<\infty$.
 Then $f\in B(X,Y)$ iff $g\in B(X,Y)$.
\item Let $f\in B(X,Y)$.  Then $\Emnp(f)<\infty$ for all $n\geq 1$.
\end{enumerate}
\end{lem}
\begin{proof}
For all $x,x'\in X$,
$$d_Y(g(x),g(x')) \leq d_Y(g(x),f(x)) + d_Y(f(x),f(x')) + d_Y(f(x'),g(x'))$$
and hence
\begin{equation}\label{eq:BXY-dist}
3^{1-p} \dY(g(x),g(x'))^p \leq \dY(g(x),f(x))^p + \dY(f(x),f(x'))^p + \dY(f(x'),g(x'))^p.
\end{equation}
Integrating $d\mu(x\to x')$ gives $\Gamma$-invariant functions of $x$
which may be integrated $d\bar\nu(x)$.  Using the stationarity of $d\bar\nu$
we then have:
$$ \Emp (g) \leq 3^{p-1} \Emp(f) + 3^{p-1} d_p (f,g)^p.$$

Similarly, integrating
$$ n^{1-p} \dY (f(x_0),f(x_n))^p \leq \sum_{i=0}^{n-1} \dY(f(x_i),f(x_{i+1}))^p$$
against $d\bar\nu(x_0) \prod_{i=0}^{n-1} d\mu(x_i \to x_{i+1})$
gives $\Emnp(f) \leq n^{p-1} \Emp(f)$.
\end{proof}

Continuing the analysis of the map $x\mapsto f_*\mu_x$, we note that
this is a $\G$-equivariant map $X\to \MY$.  Since $\G$ acts by
isometries, the map
$$
\left(\Amp f\right)(x) \eqdef \cp(f_*\mu_x) $$
 is also
$\G$-equivariant; this will be our averaging procedure.  If $Y$ is a Hilbert
space and $p=2$, $\Amp$ is the usual linear average with respect to $\mu$.
In particular, $A^{(2)}_{\mu_1}A^{(2)}_{\mu_2} = A^{(2)}_{\mu_1*\mu_2}$.
This does not hold in general (for spaces other than Hilbert space, or
even in Hilbert space for $p>2$).  In particular, we will later use $\Amnp$
for large $n$ and not just $\left(\Amp\right)^n$.

We first verify that the averaging procedure remains in the space $B(X,Y)$.
\begin{lem}\label{lem:avg-control}
For $f\in B(X,Y)$ we have
$$d_p(f,\Amp f) \lesssim_{\uY} \left( \Emp(f) \right)^{1/p},$$
$$\Emp(\Amp f)\lesssim_{p,\uY} \Emp(f).$$
\end{lem}
\begin{proof}
At every $x\in X$ the fundamental estimate~\eqref{eq:dsigma-growth}
gives:
\begin{equation*}
\uYp d_p\left(f(x),\Amp f(x)\right)^p   \leq
d_p(f(x),f_*(\mu_x)^p
 = \int \dY(f(x),f(x'))^pd\mu(x\to x').
\end{equation*}
Now both sides are $\Gamma$-invariant functions of $x\in X$ and the
first claim follows by integrating against $d\bar\nu$. For the
second claim apply inequality~\eqref{eq:BXY-dist} from the proof of
Lemma~\ref{lem:bound-energies} with $g=\Amp(f)$.
\end{proof}

We measure the contractivity of $\Amp$ with respect to the energy
$\Emp$.  It is not hard to verify that contraction will imply
the existence of fixed points:

\begin{prop}\label{pro:iter-fp}
Assume that there exist $n\geq 1$ and $c<1$ such that for all $f \in
B(X,Y)$ we have $\Emp ( \Amnp f ) \leq c \Emp (f)$. Suppose that the
graph on $X$ given by connecting $x,x'$ if $\mu(x\to x')>0$ is
\emph{connected}.  Then, as long as $B(X,Y)$ is non-empty (this is
the case, for example, when $\Gamma\bs X$ is finite), it contains
constant maps.  In particular, $\Gamma$ fixes a point in $Y$.
\end{prop}
\begin{proof}
Choose an arbitrary $f_0 \in B(X,Y)$ and let $f_{k+1} = \Amnp f_k$.
By assumption we have $\Emp(f_k) \leq c^k \Emp(f_0)$. By Lemma
\ref{lem:bound-energies} $\Emnp (f_k) \leq n^{p-1} c^k \Emp(f_0)$,
and by Lemma \ref{lem:avg-control} this means that
$$d_p(f_{k+1},f_k)^p \lesssim_{p,\uY,n} c^k \Emp(f_0).$$
It now follows that $f_k$ are a Cauchy sequence and hence converge
to a function $f \in B(X,Y)$.  We have $\Emp(f) = 0$ so $f(x) =
f(x')$ whenever $\mu(x\to x')>0$.  By the connectivity assumption
this means $f$ is constant on $X$ and its value is the desired fixed
point.
\end{proof}

We now address the problem of showing that averaging reduces the
energy. We prove two technical inequalities:
\begin{prop}(generalization of \cite[B.25]{Silberman:RandGpT})
\label{pro:cancellation} We have,
$$\Emp \left(\Amnp f\right) \lesssim_{p,\uY}
\int_{\GX} d\bn(x) \int_X \left[ d\mu^{n+1}(x\to x')- d\mu^{n}(x\to x')\right]
   \dY\left(\Amnp f(x), f(x')\right)^p. $$
\begin{equation}\label{eq:energy} \int_{\GX} d\bn(x) \int_X d\mu^{n}(x\to x')
      \dY\left(\Amnp f(x), f(x')\right)^p
 \lesssim_{p,c_Y} \Emnp (f).\end{equation}
\end{prop}
\begin{proof}
Recall that $\Amnp f(x) = \cp(f_*\mu^n(x\to\cdot))$. The fundamental
estimate~\eqref{eq:dsigma-growth} then reads:
\begin{multline}\label{eq:this} \uYp \dY\left(y,\Amnp f(x)\right)^p\\ \leq
\int_X \dY(y,f(x'))^pd\mu^n(x\to x')
 - \int_X \dY\left(\Amnp f(x), f(x')\right)^p d\mu^n(x\to x').
\end{multline}
Setting $y = \Amnp f(x'')$ integrate~\eqref{eq:this} $d\mu(x''\to x)$. The
resulting function of $x''$ is $\Gamma$-invariant and we integrate
it $d\bn(x'')$ to get (also using the reversibility),
\begin{multline*}
2\uYp \Emp (\Amnp f) \leq \int_\GX d\bn(x'') \int_X
d\mu^{n+1}(x''\to x') \dY(\Amnp f(x''),f(x'))^p \\  - \int_\GX
d\bn(x) \int_X d\mu^{n}(x\to x'') \dY(\Amnp f(x''),f(x))^p.
\end{multline*}
Inequality~\eqref{eq:energy} follows directly from the triangle
inequality and Lemma \ref{lem:avg-control}.
\end{proof}

\begin{thm}\label{thm:poincare-to-avg}
Let $\G$ be a group generated by the symmetric set $S$ of size $2d$,
acting by isometries on the $p$-uniformly convex space $Y$, let $X =
\Cay{\G}$ (the Cayley graph of $\G$), and let $f\in B(X,Y)$.  Let
$\mu$ be the $j$th convolution power of the standard random walk on
$X$ for an even $j$.  Then
$$ \Emp\left(\Amnp f\right) \lesssim_{p,\uY,d,j} \sqrt\frac{\log n}{n}\cdot \Emnp(f) +
   \frac{1}{n}\cdot \Emp (f).$$
\end{thm}
\begin{proof}
Pulling back $f$ to a function on the free group on $S$ (acting on
$Y$ via the quotient map) we may assume that $\Gamma$ is the free
group and $X$ the $2d$-regular tree. Now, \cite[Prop.\
2.9]{Silberman:RandGpT} implies that $\mu^{n+1}(x\to
x')-\mu^{n}(x\to x')$ is typically small: given $x$, except for a
set of $x'$ of $(\mu^{n+1}+\mu^{n})(x\to\cdot)$-mass $\lesssim_d
n^{-\theta}$, the difference is $\lesssim_{d,j,\theta}
\sqrt\frac{\log n}{n} \mu^n(x\to x'),$ where $\theta > 0$ is
arbitrary.

Applying this in Proposition \ref{pro:cancellation} we find that:
$$\Emp (\Amnp f) \lesssim_{p,\uY,d,j,\theta} \sqrt\frac{\log n}{n}\cdot \Emnp (f)
 + n^{-\theta} \max_{\substack{d_X(x,x'')\leq j(n+1)\\2|d_X(x,x'')}} \dY(\Amnp f(x),f(x''))^p.$$

Now
$$\dY(\Amnp f(x),f(x'')^p) \lesssim_{p,\uY}
 \max_{\substack{d_X(x,x')\leq jn\\2|d_X(x,x')}} \dY(f(x'),f(x''))^p,$$
and by the triangle inequality
$$\dY(f(x'),f(x''))^p \leq (2n+1)^{p-1}
    \max_{\substack{d_X(x,x') \leq j\\2|d_X(x,x')}} \dY(f(x),f(x'))^p.$$
Finally, the latter quantity is at most $\lesssim_{j,d} \Emp(f)$ (one
needs that $\mu^j(x\to\cdot)$ is supported on all points $x'$ at
even distance from $x$ at most $j$).  Putting it all together we
have:

$$ \Emp (\Amnp f) \lesssim_{p,\uY,d,j,\theta}
     \sqrt\frac{\log n}{n} \cdot \Emnp (f)
    + n^{p-1-\theta} \Emp (f),$$
    as required.
\end{proof}

It is now clear that (assuming we can arrange $n$ to be large)
what is needed is that $\Emnp(f)$ is not too large compared to $\Emp(f)$.
This is what we establish in the next two sections.

\section{Poincar\'e inequalities on metric spaces}\label{sec:poincare}

It turns out that it is hard to show directly that averaging with
respect to the generators of the random group reduces the energy
(compare \cite{IzekiNayatni:FP_RandGp}).  Instead, it is preferable
to average with respect to some power of the generators, as in
Theorem~\ref{thm:poincare-to-avg}, where we gain by making $n$
large. This requires controlling $\Emnp (f)$ in terms of $\Emp (f)$.
Such a control takes the form of inequalities involving distances
alone rather than centers-of-mass, so that methods from metric
embedding theory can be used to prove them. In this section we state
the inequalities the we need, and show that they hold for functions
from expander graphs to certain target metric spaces ($L_p$ spaces
and $\mathrm{CAT}(0)$ manifolds). In Section~\ref{sec:nagata} we use metric
embeddings to establish these inequalities for additional classes of
metric spaces. In section \ref{sec:mainthm} we then show that a
strong enough Poincar\'e inequality for a particular target is
enough to control averaging so that the random group has the
fixed-point property on that target.

We fix a group $\G$, a discrete $\Gamma$-space $X$, a
$\G$-equivariant random walk $\mu \in \WX$, reversible with respect
to the $\G$-invariant measure $\nu$ which gives finite measure to
any fundamental domain for $\GX$.

\begin{defn} Let $Y$ be a metric space, and  $p\geq 1$.
Let $n > m\geq 1$ be integers. We say that a \emph{Poincar\'e
inequality of exponent $p$} holds if there exists $c>0$ such that
for any $f\in B(X,Y)$,
\begin{equation}\label{eq:poincare-finite}
\Emnp(f) \leq c^p \Emmp(f).
\end{equation}
If $\nu$ itself is a probability measure, we also say that a
Poincar\'e inequalty holds if exists $\bar c>0$ such that for any
$f$,
\begin{equation}\label{eq:poincare-infty}
\int_{X\times X} d\nu(x)d\nu(x') \dY(f(x),f(x'))^p \leq {\bar c}^p
\Emmp(f).
\end{equation}
\end{defn}

Inequality~\eqref{eq:poincare-infty}, when $Y$ is a Hilbert space
and $p=2$ is the classical Poincar\'e inequality. It is sometimes
easier  to work with than the inequality~\eqref{eq:poincare-finite}
(for example when proving such results as the extrapolation lemma
below). It will be inequality~\eqref{eq:poincare-finite}, however,
that will be used for obtaining fixed point properties for the
random group. Note that inequality~\eqref{eq:poincare-infty} can be
thought of as the limit as $n\to\infty$
of~\eqref{eq:poincare-finite}.


\begin{lem}\label{lem:poincare-def} Let $\nu$ be a probability measure.
\begin{enumerate}
\item Assume that \eqref{eq:poincare-infty} holds with the constant $\bar c$.  Then
\eqref{eq:poincare-finite} holds with $c=2 \bar c$ for all $n>m$.
\item Assume that $Y$ is $p$-uniformly convex, and let
$V^{(p)} (f) = \int_X d\nu(x) \dpY(f(x),c_p(f_* \nu))$. Then
$$ V^{(p)}(f) \leq
   \int_{X\times X} d\nu(x)d\nu(x') \dpY(f(x),f(x'))
\leq 2^{p-1} V^{(p)} (f). $$
\end{enumerate}
\end{lem}
\begin{proof}
For any $x,x',x''\in X$ we raise the triangle inequality to the
$p$th power to obtain:
$$\dY(f(x),f(x'))^p \leq 2^{p-1} \dY(f(x),f(x''))^p +
2^{p-1}\dY(f(x'),f(x''))^p .$$ Integrating against
$d\nu(x)d\mu^n(x\to x')d\nu(x'')$ and using the stationarity and
reversibility of the Markov chain gives:
$$ \Emnp(f) \leq 2^{p-1} \int_{X\times X} d\nu(x)d\nu(x') \dY(f(x),f(x'))^p,$$
whence the first claim. For the proof of the second claim write $y_0
= c_p(f_* \nu)$, and recall that $\int_X d\nu(x) \dY(f(x),y_0)^p
\leq \int_X d\nu(x) \dY(f(x),y)^p$ holds for all $y \in Y$ by
definition of $c_p$. Setting $y=f(x')$ and averaging \wrt $x'$ gives
half of the inequality. For the other half use $d_Y(f(x),f(x')) \leq
d_Y(f(x),y_0) + d_Y(y_0,f(x')).$
\end{proof}

We study metric inequalities for functions on finite Markov chains
(typically, the standard random walks on finite graphs).  In the following
we use the shorthand $(V,\mu,\nu)$ for the data of a finite set $V$
(``vertices''), and a Markov chain $\mu\in\WV$ reversible with respect
to a probability measure $\nu\in\MV$.  Recall that the Markov chain is
\emph{ergodic} if for any $u,v\in V$ there is $n$ such that $\mu^n(u\to v)>0$.
For such a Markov chain the averaging operator $A_\mu^(2)$ acting on
$L^2(\nu)$ is the usual nearest-neighbour averaging operator
$Af(u)=\int_V f(v)d\mu(u\to v)$.  It is well-known that this is a self-adjoint
operator with spectrum contained in $[-1,1]$, with $1$ a simple eigenvalue
(here we use ergodicity).  The \emph{spectral gap} of the chain is then
the difference between $1$ and the second largest eigenvalue.

\begin{defn}\label{def:small} To the metric space $Y$ we associate its
\emph{Poincar\'e modulus} of exponent $p$, $p\geq 2$. Denoted
$\Pmodp_Y(\sigma)$, it is the smallest number $\Pmod$ such that for
any finite reversible ergodic Markov chain $(V,\mu,\nu)$ with
spectral gap at least $\sigma$ and any function $f\colon V\to Y$ we
have \begin{equation}\label{eq:def spectral}
 \int_{V\times V}
d\nu(u)d\nu(v) \dY(f(u),f(v))^p \leq
   \Pmod^p \int_{V\times V} d\nu(u)d\mu(u\to v)
   \dY(f(u),f(v))^p.\end{equation}
   Observe that spectrally expanding both sides of~\eqref{eq:def spectral} shows that for $Y$ Hilbert space,
$\Pmodt_Y(\sigma) = \frac{1}{\sqrt\sigma}$.

We also define the \emph{Local Poincar\'e modulus} of exponent $p$ to be
$$ \Pmodp_Y(\sigma,N) = \sup \left\{ \Pmodp_{Y'}(\sigma) \mid
    Y'\subseteq Y,\, \vert Y'\vert \leq N \right\}.$$
We say that $Y$ has \emph{small Poincar\'e moduli} of exponent $p$
if its local Poincar\'e moduli of that exponent satisfy
\begin{equation}\label{eq:1/4}
 \Pmodp_Y(\sigma,N) \lesssim_{p,\sigma} o\left(\left(\frac{\log N}{\log \log N}\right)^{\frac{1}{2p}}\right).
\end{equation}
Note that a bound of $O(\log N)$ in~\eqref{eq:1/4} holds
true for any metric space, by Bourgain's embedding
theorem~\cite{Bourgain:MetricSpaceEmbed} and~\eqref{eq:hilbert p} below.
\end{defn}

We shall proceed to bound the Poincar\'e modulus for non-Hilbertian
spaces, i.e., to show that a Poincar\'e inequality holds for Markov
chains on these spaces, with the constant bounded by a function of
the spectral gap of the chain. The first case is that of $L_p$. The
proof below is a slight variant of Matou\v{s}ek's extrapolation
lemma for Poincar\'e inequalities; see~\cite{Mat97}, and~\cite[Lem.\
5.5]{BLMN05}; we include it since it has been previously stated for graphs rather than general Markov chains.

\begin{lem}[Matou\v{s}ek extrapolation]\label{lem:Matousek}
Let $(V,\mu,\nu)$ be a reversible Markov chain.
Assume the Poincar\'e inequality \eqref{eq:poincare-infty} holds
with exponent $p\geq 1$ and Poincar\'e modulus $Ap$ for functions $f\colon
V\to \R$. Then for any $q\geq p$ the inequality
\eqref{eq:poincare-infty} holds for such functions with the exponent
$q$ and modulus $4Aq$ and for any $1<q\leq p$ the inequality holds with
exponent $q$ and modulus $Aq$.
\end{lem}
\begin{proof}
For $u\in V$ set $g(u) = \abs{f(u)}^\fqp\sgn f(u)$. Shifting $f$ by a constant
does not change the claimed inequalities, and using the intermediate value
theorem we may assume $\int gd\nu=0$.
By the convexity of the norm, H\"older's inequality, and the assumed
Poincar\'e inequality, we have:
\begin{multline*}
\norm{g}_{L_p(\nu)}=  \norm{g - \int gd\nu}_{L_p(\nu)}
 \leq       \int d\nu(v) \norm{g-g(v)}_{L_p(\nu)}
 \leq  \left( \int d\nu(v)\norm{g-g(v)}_{L_p(\nu)}^p \right)^\fop \\
 = \left( \int d\nu(u) d\nu(v) \abs{g(u)-g(v)}^p \right)^\fop
 \leq  (Ap) \left( \int d\nu(u) d\mu(u\to v) \abs{g(u)-g(v)}^p \right)^\fop.
\end{multline*}
We next use the elementary inequality
$$ \abs{a^\fqp \pm b^\fqp} \leq \fqp \abs{a\pm b}
   \left( a^{\fqp-1}+ b^{\fqp-1} \right), $$
to deduce that: \begin{eqnarray}\label{eq:holder} \nonumber
\norm{g}_{L_p(\nu)} &\leq& (Aq) \left[ \int d\nu(u)d\mu(u\to v)
\abs{f(u)-f(v)}^p
   \left(\abs{f(u)}^{\fqp-1}+\abs{f(v)}^{\fqp-1}\right)^p
   \right]^\fop\\ &\leq& (Aq) \left[\int d\nu(u)d\mu(u\to v)
   \abs{f(u)-f(v)}^q\right]^\foq \nonumber\\&&\quad\cdot
\left[\int d\nu(u)d\mu(u\to v)  \left(
    \abs{f(u)}^{\fqp-1}+\abs{f(v)}^{\fqp-1}\right)^\frac{qp}{q-p}
\right]^\frac{q-p}{pq},
   \end{eqnarray}
where we used H\"older's inequality.

By the triangle inequality in $L_{qp/(q-p)}$, symmetry and
reversibility, the last term in~\eqref{eq:holder} is at most:
$$2\left[\int d\nu(u)
   \abs{f(u)}^{\frac{q-p}{p}\cdot\frac{qp}{q-p}}\right]^\frac{q-p}{pq} =
   2\norm{f}_{L_q(\nu)}^\frac{q-p}{p}.$$
Recalling that $\abs{g(u)} = \abs{f(u)}^\fqp$, this means
$$\norm{f}_{L_q(\nu)}^\fqp \leq (2Aq)
\left[\int d\nu(u)d\mu(u\to v) \abs{f(u)-f(v)}^q\right]^\foq
\norm{f}_{L_q(\nu)}^{\fqp-1},$$ and collecting terms finally gives
$$ \norm{f}_{L_q(\nu)} \leq (2Aq)
   \left[\int d\nu(u)d\mu(u\to v) \abs{f(u)-f(v)}^q\right]^\foq.$$
To conclude the proof we note that
$$ \left[\int d\nu(u)d\nu(v) \abs{f(u)-f(v)}^q\right]^\foq \leq
    2 \norm{f}_{L_q(\nu)}$$
follows by applying the triangle inequality in $L_q(\nu\times\nu)$
to the functions $(u,v)\mapsto f(u)$ and $(u,v)\mapsto -f(u)$.
\end{proof}

\begin{cor} We have $\Pmodp_\R(\sigma) \leq 2p\frac{1}{\sqrt\sigma}$.
Integrating, this bound also holds for for $\Pmodp_{L_p}(\sigma)$.
\end{cor}
Since Hilbert space embeds isometrically into $L_p$ for all $p\ge
1$, we see that for $p\ge 2$,
\begin{equation}\label{eq:hilbert p}
\Pmodp_{L_2}(\sigma) \leq \Pmodp_{L_p}(\sigma) \leq \frac{2p}{\sqrt{\sigma}}.\end{equation}
\begin{rem} In \cite{Mat97} it is shown
that any $N$-point metric space embeds in $L_p$ with distortion
$\lesssim 1+\frac{1}{p}\log_N$.  It follows that for any metric
space $Y$ and any exponent $p\geq 2$,
$$\Pmodp_Y(\sigma,N) \lesssim \frac{p+\log N}{\sqrt{\sigma}}.$$
\end{rem}

\begin{rem}\label{rem:ozawa}
The argument above was special for $L_p$ spaces. But, using a
different method, it was shown in~\cite{Ozawa04} that for any Banach
lattice $Y$ that does not contain almost isometric copies of every
finite metric space, we have
$\Lambda_Y^{(2)}(\sigma)\lesssim_{Y,\sigma} 1$. While this is not
stated explicitly in~\cite{Ozawa04}, it follows easily from the
proof of Lemma A.4 there; this observation is carried out in detail
in~\cite{NR05}.
\end{rem}

We also note for future reference that the Poincar\'e modulus
behaves well under natural operations on metric spaces. The
(trivial) proof is omitted.
\begin{prop} Fix a function $L(\sigma,N)$ and let $\mathcal{C}$
be the class of metric spaces $Y$ such that $\Pmodp_Y \leq L$.  Then
$\mathcal{C}$ is closed under completion, passing to subspaces,
$\ell_p$ products, and ultralimits. The property of being
$p$-uniformly convex with constant $\uY$ is also preserved by the
same operations, except that that one must pass to convex (\ie
totally geodesic) subspaces.
\end{prop}

In the class of $\mathrm{CAT}(0)$ spaces, a further reduction is possible: it
is enough to establish the Poincar\'e inequality for all the
\emph{tangent cones} of the space $Y$.  This is essentially an
observation from~\cite[Pf of Thm.\ 1.1]{Wang:FixedPoints} (see also
\cite[Lem.\ 6.2]{IzekiNayatani:HarmonicMaps}). It relies on the
equivalent formulation from Lemma~\ref{lem:poincare-def}. We recall
the definition of the \emph{tangent cone} to a metric space $Y$ at
the point $y\in Y$.  Let $\gamma,\gamma'\colon:[0,\epsilon]\to Y$ be
unit-speed geodesic segments issuing from $y$.  For each $t>0$ let
$\theta_{t,t'}$ be the angle such that
$$d_Y(\gamma(t),\gamma'(t')^2 = d_Y(y,\gamma(t))^2 + d_Y(y,\gamma'(t'))^2 -
  2 d_Y(y,\gamma(t))d_Y(y,\gamma'(t')) \cos \theta_{t,t'}.$$
The \emph{Alexandroff angle} between $\gamma,\gamma'$ is defined as
$\theta(\gamma,\gamma') = \limsup_{t,t'\to 0} \theta_{t,t'}$.  It is
easy to check that this provides a pesudometric on the space of
germs of geodesic segments issuing from $y$.  Identifying segments
at angle zero gives the \emph{space of directions} $S_y Y$. Now let
$T_yY$ be the infinite cone over $S_y Y$ with the metric
$\tilde{d}(a \gamma, b\gamma') = \sqrt{a^2 + b^2 - 2ab
\cos\theta(\gamma,\gamma')}$. There is a natural ``inverse of the
exponential map'' $\pi_y \colon Y\to T_y Y$ given by mapping $z\in
Y$ to $d_Y(y,z)\cdot [y,z]$ where $[y,z]$ is the geodesic segment
connecting $y$ to $z$ ($\pi_y(y)$ is the cone point).  By definition
$\pi_y$ preserves distances from $y$, in that
$\tilde{d}(\pi_y(y),\pi_y(z)) = d_Y(y,z)$. The key properties for us
are that when $Y$ is $\mathrm{CAT}(0)$, $\pi_y$ is 1-Lipschitz (in fact, this
is equivalent to the $\mathrm{CAT}(0)$ inequality) and that in that case
$(T_y Y,\tilde{d})$ is a $\mathrm{CAT}(0)$ metric space as well. \cite[Thm.\
II.3.19]{BridsonHaelfliger:NonPosSpc}.   It is also important to
note that if $\sigma$ is a probability measure on $Y$ and $y =
\ct(\sigma)$ then $\ct(\pi_{y*} \sigma) = \pi_y(y)$ (this is since
the fact that $y$ minimizes $z\mapsto d_Y(\sigma,z)$ can be stated
in terms of distances from $y$ alone; see \cite[Prop.\
3.5]{IzekiNayatani:HarmonicMaps}).

The following proposition was proved in an equivalent form
in~\cite{Wang:FixedPoints}.
\begin{prop}\label{pro:wang-trick}
Let $Y$ be a $\mathrm{CAT}(0)$ space.  Then
$$\Pmodt_Y(\sigma,N) \le 2 \sup_{y\in Y} \Pmodt_{T_y Y}(\sigma,N).$$
In particular, $\Pmodt_Y(\sigma) \leq 2 \sup_{y\in Y} \Pmodt_{T_y Y}(\sigma)$.
\end{prop}
\begin{proof}
Let $(V,\mu,\nu)$ be a finite Markov chain as above.  For a $\mathrm{CAT}(0)$
space $Y$ let $v(Y)$ be minimal such that for all $f\colon V\to Y$
we have
$$ V^{(2)}(f) \leq v(Y) \Emmt(f).$$
Lemma \ref{lem:poincare-def} shows that the constant $c$ in the
Poincar\'e inequality for functions from $V$ to $Y$ satisfies $v(Y)
\leq c \leq 2v(Y)$.  It thus remains to show that $v(Y) \leq
\sup_{y\in Y} v(T_y Y)$.  Indeed, let $f\colon X\to Y$, and let $y =
\ct(f_*\nu)$, $\tilde{f} = \pi_y \circ f$.  As noted above we have
$\ct(\tilde{f}_*\nu) = \pi_y(y)$, and since distances from $y$ are
preserved that $V^{(2)} (f) = V^2(\tilde{f})$. Since $\pi_y$ is
non-expansive, $\Emmt(\tilde{f}) \leq \Emmt(f)$. It follows that
$V^{(2)} (f) \leq v(T_{\ct(f_*\nu)}) \Emmp(f)$ and we are done.
\end{proof}

Note that when $Y$ is a Riemannian manifold, the tangent cone constructed
above is isometric to the ordinary tangent space at $y$, equipped with the
inner product given by the Riemannian metric at that point.  in other words,
the tangent cones of a manifold are all isometric to Hilbert spaces.
An approximation argument also shows that
$\Pmodt_Y(\sigma) \geq \Pmodt_{T_y Y}(\sigma)$ for all $y\in Y$.

\begin{cor}
Let $Y$ be a Hilbert manifold with a $\mathrm{CAT}(0)$ Riemannian metric
(for example, a finite-dimensional simply connected Riemannian manifold
of non-positive sectional curvature).  Then
$\frac{1}{\sqrt{\sigma}} \leq \Pmodt_Y(\sigma) \leq \frac{2}{\sqrt\sigma}$.
\end{cor}

\section{Padded decomposability and Nagata
dimension}\label{sec:nagata}

We start by recalling some definitions and results
from~\cite{LN05-via}. Let $(X,d_X)$ be a metric space.
Given a partition $\mathscr P$ of $X$ and $x\in X$ we denote by
$\mathscr P(x)$ the unique element of $\mathscr P$ containing $x$.
For $\Delta>0$, a distribution $\Pr$ over partitions of $X$ is
called a $\Delta$-bounded stochastic decomposition if
$$
\Pr\left[\forall\  C\in \mathscr P,\ \diam(C)\le \Delta \right]=1,
$$
i.e., almost surely with respect to $\Pr$ partitions of $X$ contain
only subsets whose diameter is bounded by $\Delta$. Given
$\e,\delta>0$ we shall say that a $\Delta$-bounded stochastic
decomposition $\Pr$ is $(\e,\delta)$-padded if for every $x\in X$,
$$
\Pr\left[\mathscr P(x)\supseteq B_X(x,\e\Delta)\right]\ge \delta.
$$
Here, and in what follows, $B_X(x,r)\eqdef \{y\in X:\ d_X(x,y)\le
r\}$ denotes the closed unit ball of radius $r$ centered at $x$.

Given two metric spaces $(Y,d_Y)$ and $(Z,d_Z)$, and $X\subseteq Y$,
we denote by $e(X,Y,Z)$ the infimum over all constant $K$ such that
every Lipschitz function $f:X\to Z$ can be extended to a function
$\tilde f:Y\to Z$ such that $\|\tilde{f}\|_{\Lip}\le
K\cdot \|f\|_{\Lip}$. The {\em absolute Lipschitz extendability}
constant of $(X,d_X)$, denoted $\Ae(X)$, is defined as
$$
\Ae(X)\eqdef \sup\left\{e(X,Y,Z):\ Y\supseteq X,\ Z\ \text{a
Banach space}\right\}.
$$
In words, the inequality $\Ae(X)<K$ implies that any Banach space
valued Lipschitz mapping on $X$ can be extended to any metric space
containing $X$ such that the Lipschitz constant of the extension
grows by at most a factor of $K$. This notion was introduced
in~\cite{LN05-via}, where several classes of spaces
were shown to be absolutely extendable. We note that in the
extension theorems we quote below from~\cite{LN05-via} the role of
the target space being a Banach space is very weak, and it can also
be, for example, any $\mathrm{CAT}(0)$ space; we refer to~\cite{LN05-via}
for a discussion of this issue.

The following theorem was proved in~\cite{LN05-via}.

\begin{theorem}[Absolute extendability
criterion~\cite{LN05-via}]\label{thm:from inv} Fix $\e,\delta\in
(0,1)$ and assume that $(X,d_X)$ admits a $2^k$-bounded
$(\e,\delta)$-padded stochastic decomposition for every $k\in \Z$.
Then
$$
\Ae(X)\lesssim \frac{1}{\e\delta}.
$$
\end{theorem}

In~\cite{LN05-via} several classes of spaces were shown to satisfy
the conditions of Theorem~\ref{thm:from inv}, including subsets of
Riemannian surfaces of bounded genus and doubling metric spaces. For
our applications we need to enrich the repertoire of these spaces.
We do so by relating the notion of padded decomposability to having
finite Nagata dimension, and using results from~\cite{LS05} which
bound the Nagata dimension of various classes of spaces (which will
be listed shortly).

Let $(X,d_X)$ be a metric space. Following~~\cite{Nagata83,LS05}, given $\gamma\ge 1$ and $d\in
\mathbb N$ we say that $X$ has Nagata dimension at most $d$
with constant $\gamma$ if for every $s>0$ there exists a family of
subsets $\mathscr C\subseteq 2^X\setminus\{\emptyset\}$ with the
following properties.

\begin{enumerate}
\item $\mathscr C$ covers $X$, i.e. $\bigcup_{C\in \mathscr C} C=X$.
\item For every $C\in\mathscr C$, $\diam(C)\le \gamma s$.
\item For every $A\subseteq X$ with $\diam(A)\le s$, we have $ \left|\left\{C\in \mathscr C:\ C\cap A\neq \emptyset \right\}\right|\le
d+1$.
\end{enumerate}
The infimum over all $\gamma$ for which $X$ has Nagata dimension at
most $d$ with constant $\gamma$ will be denoted $\gamma_d(X)$. If no
such $\gamma$ exists we set $\gamma_d(X)=\infty$. Finally,
the Nagata dimension of $X$ is defined as
$$
\dim_N(X)=\inf\left\{d\ge0: \ \gamma_d(X)<\infty\right\}.
$$
It was proved in~\cite{LS05} that $X$ has finite Nagata dimension if
and only if $X$ embeds quasisymmetrically into a product of finitely
many trees.

\begin{lem}[Bounded Nagata dimension implies padded
decomposability] Let $(X,d_X)$ be a metric space, $\gamma\ge 1$ and
$d\in \mathbb N$. Assume that $\gamma_d(X)< \gamma<\infty$. Then for
every $k\in \Z$, $X$ admits a $2^k$-bounded
$\left(\frac{1}{100\gamma d^2},\frac{1}{d+1}\right)$-padded
stochastic decomposition.
\end{lem}

\begin{proof} It is easy to iterate the definition of Nagata
dimension to prove the following fact, which is (part of)
Proposition 4.1 in~\cite{LS05} (with explicit, albeit sub-optimal,
estimates, that can be easily obtained from an examination of the proof
in~\cite{LS05}). Let $r=50\gamma\cdot d^2$. For every $j\in \Z$
there exists a family of subsets $\B\subseteq
2^{X}\setminus\{\emptyset\}$ with the following properties.
\begin{enumerate}
\item For every $x\in X$ there exists $B\in \B$ such that
$B_X\left(x,r^j\right)\subseteq B$.
\item $\B=\bigcup_{i=0}^d \B_i$, where for every $i\in
\{0,\ldots,d\}$ the sets in $\mathscr B_i$ are disjoint, and for
every $B\in \B_i$, $\diam(B)\le r^{j+1}$.
\end{enumerate}

We now construct a random partition $P$ of $X$ as follows. Let $\pi$
be a  permutation of $\{0,\ldots, d\}$ chosen uniformly at random from all such $(d+1)!$ permutations.
Define a family of subsets $\widetilde \B^\pi_i \subseteq 2^{X}\setminus \{\emptyset\}$
inductively as follows: $\widetilde \B_0^\pi =\B_{\pi(0)}$, and for
$0\le i< d$,
$$
\widetilde \B_{i+1}^\pi=\left\{B\setminus
\bigcup_{C\in\bigcup_{\ell=0}^i\widetilde \B_{\pi(\ell)}^\pi}C:\
B\in \B_{\pi(i+1)}\right\}\setminus\{\emptyset\}.
$$
Finally we set $\mathscr P^\pi=\bigcup_{i=0}^d \widetilde \B_i^\pi$.
Since $\B$ covers $X$, $\mathscr P^\pi$ is a partition of $X$.
Moreover, by construction, for every $C\in \mathscr P^\pi$,
$\diam(C)\le r^{j+1}$.

Fix $x\in X$. By the first condition above there exists $i\in
\{0,\ldots,d\}$ and $B\in \B_i$ such that
$B_X\left(x,r^j\right)\subseteq B$. If $\pi(0)=i$ then
$\mathscr P^\pi(x)=B\supseteq B(x,r^j)$. This happens with probability
$\frac{1}{d+1}$.

Letting $k$ be the largest integer $j$ such that $r^{j+1}\le 2^k$ we
see that $\mathscr P^\pi$ is a $2^k$-bounded stochastic partition
such that for every $x\in X$
$$
\Pr\left[\mathscr P^\pi (x)\supseteq
B\left(x,\frac{2^{k-1}}{r}\right)\right]\ge \frac{1}{d+1},
$$
as required.
\end{proof}

The following corollary shows that many of the Lipschitz extension
theorems proved in~\cite{LS05} are direct conequences of the earlier results
of~\cite{LN05-via}. The cubic dependence on the Nagata dimension is
an over-estimate, and can be easily improved. We believe that
the true bound should depend linearly on the dimension, but this is
irrelevant for the purposes of the present paper.

\begin{corollary}\label{coro:list} For every metric space $X$ and $d\in \mathbb N$,
$$
\Ae(X)=O\left(\gamma_d(X)d^3\right).
$$
Thus, doubling metric spaces, subsets of compact Riemannian
surfaces, Gromov hyperbolic spaces of bounded local geometry,
Euclidean buildings, symmetric spaces, and
homogeneous Hadamard manifolds, all have finite absolute
extendability constant.
\end{corollary}
The list of spaces presented in Corollary~\ref{coro:list} is a
combination of the results of~\cite{LN05-via} and~\cite{LS05}. In
particular the last four classes listed in Corollary~\ref{coro:list}
were shown in~\cite{LS05} to have finite Nagata dimension. It should
be remarked here that Lipschitz extension theorems for Gromov
hyperbolic spaces of bounded local geometry were previously proved
in~\cite{NPSS04} via different methods.

We will use the following embedding theorem, which follows
from the proof of Theorem 5.1 in~\cite{LMN05}, though it isn't
explicitly stated there in full generality. We include the simple
proof for the sake of completeness.

\begin{thm}[Snowflake embedding]\label{thm:snowflake} Fix $\e,\delta,\theta\in (0,1)$. Let
$(X,d_X)$ be a metric space which admits for every $k\in \Z$ a
$2^k$-bounded $(\e,\delta)$-padded stochastic decomposition. Then the metric space $\left(X,d_X^\theta\right)$ embeds into Hilbert space with bi-Lipschitz distortion
$\lesssim\frac{1}{\e\sqrt{\delta
\theta(1-\theta)}}.
$
\end{thm}

\begin{proof} For every $k\in \Z$ let $\Pr_k$ be an $(\e,\delta)$-padded distribution over
$2^k$-bounded partitions of $X$. We also let
$\{\sigma_C\}_{C\subseteq X}$ be i.i.d. symmetric $\pm 1$ Bernoulli
random variables, which are independent of $\Pr_k$. Denote by
$\Omega_k$ the measure space on which all of these distributions are
defined. Let $f_k:X\to L_2(\Omega_k)$ be given by the random
variable
$$
f_k(x)=\sigma_{\mathscr P(x)}\cdot \min\left\{d_X\left(x,X\setminus
\mathscr P(x)\right),2^k\right\}\quad (\text{$\mathscr P$ is a
partition of $X$}).
$$
Finally, define $F:X\to \left(\bigoplus_{k\in \Z}
L_2(\Omega_k)\right)\otimes \ell_2$ by
$$
F(x)=\sum_{k\in Z} 2^{-k(1-\theta)}f_k(x)\otimes e_k.
$$

Fix $x,y\in X$ and let $k\in \Z$ be such that $2^k<d_X(x,y)\le
2^{k+1}$. It follows that for every $2^k$-bounded partition
$\mathscr P$ of $X$, $\mathscr P(x)\neq \mathscr P(y)$. Thus
$\sigma_{\Prob(x)}$ and $\sigma_{\Prob(y)}$ are independent random
variables, so that
\begin{eqnarray}\label{eq:lowerF}
&&\!\!\!\!\!\!\!\!\!\!\!\!\|F(x)-F(y)\|_2^2\ge \nonumber
2^{-2k(1-\theta)}\left\|f_k(x)-f_k(y)\right\|_{L_2(\Omega_k)}^2\\
\nonumber&=&\frac{\E_\sigma\E_{\Pr_k} \left[\sigma_{\mathscr
P(x)}\cdot\min\left\{ d_X\left(x,X\setminus \mathscr
P(x)\right),2^k\right\}-\sigma_{\mathscr P(y)}\cdot
\min\left\{d_X\left(y,X\setminus
\mathscr P(y)\right),2^k\right\}\right]^2}{2^{2k(1-\theta)}}\\
&\stackrel{(\clubsuit)}{=}& \nonumber \frac{\E_{\Pr_k}\left[\min\left\{
d_X\left(x,X\setminus \mathscr
P(x)\right)^2,2^{2k}\right\}\right]+\E_{\Pr_k}\left[\min\left\{
d_X\left(y,X\setminus \mathscr
P(y)\right)^2,2^{2k}\right\}\right]}{2^{2k(1-\theta)}}\nonumber\\
&\stackrel{(\spadesuit)}{\ge}& \frac{\delta \left(\e 2^k\right)^2}{2^{2k(1-\theta)}}
\ge\frac{\e^2\delta}{2^{2\theta}}\cdot d_X(x,y)^{2\theta},
\end{eqnarray}
where in $(\clubsuit)$ we used the independence of $\sigma_{\Prob(x)}$ and $\sigma_{\Prob(y)}$, and in $(\spadesuit)$ we used the $(\e,\delta)$-padded property.

In the reverse direction, for every $j\in \Z$, if $\Prob$ is a
$2^j$-bounded partition of $X$ then it is straightforward to check
that for all $x,y\in X$ we have the point-wise inequality,
\begin{multline}\label{eq:cases}
\left|\sigma_{\mathscr P(x)}\cdot \min\left\{d_X\left(x,X\setminus
\mathscr P(x)\right),2^j\right\}-\sigma_{\mathscr P(y)}\cdot
\min\left\{d_X\left(y,X\setminus \mathscr
P(y)\right),2^j\right\}\right|\\\le
2\min\left\{d_X(x,y),2^{j}\right\}.
\end{multline}
Indeed, if $d_X(x,y)\ge 2^j$ then~\eqref{eq:cases} is trivial.  If $\mathscr P(x)=\mathscr P(y)$ then~\eqref{eq:cases} follows from the Lipschitz condition $|d_X\left(x,X\setminus
\mathscr P(x)\right)-d_X\left(y,X\setminus
\mathscr P(x)\right)|\le d_X(x,y)$. Finally, if $d_X(x,y)<2^j$ and $\mathscr P(x)\neq\mathscr P(y)$ then $d_X\left(x,X\setminus
\mathscr P(x)\right),d_X\left(y,X\setminus
\mathscr P(y)\right)\le d_X(x,y)<2^j$, implying~\eqref{eq:cases} in this case as well.

It follows from~\eqref{eq:cases} that
\begin{multline}\label{eq:upperF}
\|F(x)-F(y)\|_2^2\lesssim \sum_{j\in \Z}\frac{
\min\left\{d_X(x,y)^2,4^{j}\right\}}{4^{j(1-\theta)}} \lesssim
\sum_{j\le k}4^{j\theta}+d_X(x,y)^2\sum_{j\ge k+1}4^{-j(1-\theta)}\\
\lesssim \frac{4^{k\theta}}{\theta}+d_X(x,y)^2\cdot
\frac{4^{-k(1-\theta)}}{1-\theta}
\lesssim \frac{d_X(x,y)^{2\theta}}{\theta(1-\theta)}.
\end{multline}
Combining~\eqref{eq:lowerF} and~\eqref{eq:upperF}, we get that the bi-Lipschitz distortion of $f$ is
$\lesssim \frac{1}{\e\sqrt{\delta \theta(1-\theta)}}$.
\end{proof}

\begin{cor}\label{cor:nagata-to-poincare}
Let
$(Y,d_Y)$ be a metric space which admits for every $k\in \Z$ a
$2^k$-bounded $(\e,\delta)$-padded stochastic decomposition (thus, $(Y,d_Y)$ can belong to one of the classes of spaces listed in Corollary~\ref{coro:list}). Then, using the notation of Section~\ref{sec:poincare}, for every $p\in [1,\infty)$ we have
\begin{equation}\label{eq:poin from embedding}
\Pmodp_Y(\sigma)\lesssim_{\e,\delta,p,\sigma} 1.
\end{equation}
\end{cor}
\begin{proof} By Theorem~\ref{thm:snowflake} the metric space $\left(Y,\sqrt{d_Y}\right)$ embeds into Hilbert space with distortion $\lesssim_{\e,\delta} 1$. By~\eqref{eq:hilbert p} we know that $\Pmod_{L_2}^{(2p)}(\sigma)\lesssim_p \sigma^{-1/2}$. It follows that $\Pmodp_Y(\sigma)\lesssim_{\e,\delta,p}\sigma^{-1} \lesssim_{\e,\delta,p,\sigma} 1$, as required.
\end{proof}

\begin{rem} For our purposes the dependence on $\sigma$ in~\eqref{eq:poin from embedding} is irrelevant. Nevertheless, the proof Corollary~\ref{cor:nagata-to-poincare} can be optimized as follows. For $\theta\in (0,1)$, use Theorem~\ref{thm:snowflake} to embed the metric space $(Y,d_Y^\theta)$ into Hilbert space with distortion $\lesssim \frac{1}{\e\sqrt{\delta \theta(1-\theta)}}$. Since by~\eqref{eq:hilbert p} we know that $\Pmod_{L_2}^{(p/\theta)}(\sigma)\lesssim \frac{p}{\theta\sqrt{\sigma}}$. Thus, there exists a universal constant $c>1$ such that
\begin{equation}\label{eq:theta}
\Pmodp_Y(\sigma)\le \left(\frac{p}{\theta\sqrt{\sigma}}\cdot \frac{c}{\e\sqrt{\theta(1-\theta)}}\right)^{1/\theta}.
\end{equation}
One can then choose $\theta$ so as to minimize the right hand side of~\eqref{eq:theta}. If one cares about the behavior of our bound as $\sigma\to 0$, then the optimal choice is $\theta=1-\frac{\log\log(1/\sigma)}{\log(1/\sigma)}$, yielding for $\sigma\in (0,1/4)$, the estimate
\begin{equation}\label{eq:optimized}
\Pmodp_Y(\sigma)\lesssim_{\e,\delta,p} \frac{\log(1/\sigma)}{\sqrt{\sigma}}.
\end{equation}
Using the ideas presented here more carefully, the logarithmic term in~\eqref{eq:optimized} was subsequently removed in~\cite{NR05} (where the dependence on $\sigma$ was of importance for certain applications).
\end{rem}

\section{A brief review of the  construction of the random group}\label{sec:graphmodel}

We recall here the ``graph model'' for random groups and
the iterative construction of a group from an appropriate sequence
of graphs.  The construction is due to Gromov \cite{Gromov:RandGp};
further details may be found in the works of Ollivier
\cite{Ollivier:RandGpSurvey,Ollivier:GraphRandGp},  or in the more recent
work of Arzhantseva-Delzant \cite{ArzhantsevaDelzant:ExampleRandGps_preprint}.

Let $G=(V,E)$ be an \emph{undirected} simple graph. The set of
edges $E$ then has a natural double cover, the set of
\emph{oriented edges} of $G$
$$\vE=\left\{ (u,v),(v,u) \mid \{u,v\}\in E\right\}.$$

Now let $\G$ be a group.  A \emph{symmetric $\G$-labeling} of $G$
is a map $\alpha\colon\vE\to \G$ such that
$\alpha(u,v)=\alpha(v,u)^{-1}$ for all $\{u,v\}\in E$. The set of
these will be denoted $\AGG$.   More generally an $S$-labeling is
a labeling whose image lies in a (symmetric) subset $S\subseteq
\G$. The set of such labels will be denotes $\AGS$.

 Let $S \subseteq \G$ be a symmetric subset, $1\notin S$. The
\emph{Cayley graph} $\CayG$ is the graph with vertex set $\G$ and
directed edge set $\left\{ (x,xs) \mid x\in\G,s\in S\right\}$.
This is actually an undirected graph since $S$ is symmetric and carries the
natural symmetric labeling $\alpha(x,xs) = s$.
The Cayley graph $\CayG$ is connected iff $S$ generates $\G$.  In
that case let $\vec{c}$ be an oriented cycle (that is, a closed path)
in that graph, and let $w \in S^*$ be the word in $S$ read along the cycle.
It is
clear that $w$ is trivial as an element of $\G$.  Conversely, any
relator $w\in S^*$ for $\G$ induces many closed cycles on $\CayG$:
starting at any $x\in\G$ one follows the edges labeled by
successive letters in $w$.  Since $w=1$ in $\G$, this path is a
closed cycle in the Cayley graph. This observation motivates the
following construction.

Given a symmetric $\Gamma$-labeling $\alpha\in\AGG$ and an oriented path
$\vp=(\ve_1, \ldots, \ve_r)$ in $G$, we set
$\alpha(\vp)=\alpha(\ve_1)\cdot\ldots\cdot\alpha(\ve_r)$. We write
$$ R_{\alpha} =
\left\{\alpha(\vec{c})\mid\vec{c}\,\textrm{ a cycle in
}G\right\},$$ and will consider groups of the form
\begin{equation}
\Ga = \G / \left<R_\alpha\right>^\mathrm{N},
\end{equation}
where $\left<R_\alpha\right>^\mathrm{N}$ is the normal closure of $\left<R_\alpha\right>$.
Alternatively, given a presentation $\G = \left<S\vert R\right>$ we also have
$\Ga = \left<S\vert R\cup R_\alpha\right>$ once we write the
labels $\alpha(\ve)$ as words in $S$. Given $u\in V(G)$ and $x\in \Cay{\Ga}$ we define a map $\alpha_{u\to x}\colon G\to \Cay{\Ga}$ as follows. For $v\in V(G)$ choose a path $\vp$ from $u$ to $v$ in $G$, and define $\alpha_{u\to x}(v)=x\alpha(\vp)$. Note that by construction, $\alpha_{u\to x}(v)$ does not depend on the choice of the path $\vp$, and hence $\alpha_{u\to x}$ is well defined.

With a choice of a probability measure $\Pr$ on $\AGG$, the groups $\Ga$
become ``random groups''.  Note the ad-hoc nature of this construction: it is very useful for proving the existence of groups with desired
properties (for example see \cite{OllivierWise:PropTInfOut}).
However, the groups $\Ga$ are not ``typical" in any sense of the word.


As above, let $S$ be a symmetric set of generators for $\G$.  For
any integer $j$ let $\Pr_j$ on $\AGj$ be given by independently
assigning a label to each edge, uniformly at random from $S^j$.
Fixing an \emph{orientation} of $E$ (\ie a section $\iota\colon
E\to\vE$ of the covering map $\vE\to E$) shows that that $\AGj$ is
non-canonically isomorphic to the product space $E^{S_j}$ and
identifies $\Pr_j$ with the natural product measure on that space.

\begin{defn}(\cite[Def.\ 50]{Ollivier:RandGpSurvey})
A sequence of finite connected graphs $\{G_i\}_{i=1}^\infty$ is
called \emph{good for random quotients} if there exist positive
constants $C,\Delta$ such that:
\begin{enumerate}
\item The maximum degree of $G_i$ satisfies $\Delta(G_i) \leq \Delta$.
\item The girth of $G_i$ satisfies $g(G_i) \geq C \cdot \diam(G_i)$
\item $\size{V(G_i)}$ (equivalently, $g(G_i)$) tend to $\infty$ with $i$.
\end{enumerate}
\end{defn}

\begin{thm}\label{thm:ollivier}
(\cite[Thm.\ 51]{Ollivier:RandGpSurvey}, \cite[Thm.\ 6.3]{ArzhantsevaDelzant:ExampleRandGps_preprint})
Let $\{G_i\}_{i=1}^\infty$
be good for random quotients, let $\G$ be a non-elementary
torsion-free hyperbolic group with property (T), and let $\epsilon
> 0$. Then there exist $A>0$, an integer $j \geq 1$ and a
subsequence $\{i_k\}_{k\geq 1}$ such that for $G = \bigsqcup_{k\ge 1}
G_{i_k}$ and $\alpha$ chosen from $\AGj$ we have with positive
$\Pr_j$-probability that:
\begin{enumerate}
\item
  For any $K\geq 1$ if we set $G_{(K)} = \bigsqcup_{k\leq K} G_{i_k}$
  and $\alpha_{(K)} = \alpha\restrict_{G_{(K)}}$,
  then $\G_{(K)} = \G_{\alpha_{(K)}}$ is a torsion-free non-elementary
  hyperbolic group.  In particular, $\Ga$ is an infinite group.
\item\label{enu:RQ-inject}
  For any choice of vertices $u_0,v,w\in V(G_{i_k})$ and $x_0\in\CayGa$
  the natural map
  $\alpha_{u_0\to x_0}\colon G_{i_k}\to X_\alpha \eqdef \CayGa$ has
$$ A\left(d_{G_{i_k}}(v,w) - \epsilon \diam (G_{i_k})\right) \leq \frac{1}{j}d_{X_\alpha}
  (\alpha_{u_0\to x_0}(v), \alpha_{u_0\to x_0}(w)) \leq d_{G_{i_k}}(v,w). $$
\end{enumerate}
\end{thm}

When we apply Theorem~\ref{thm:ollivier} in Section~\ref{sec:mainthm}, we will take the initial group $\Gamma$ to be a free group. Even though $\Gamma$ does not have property (T), Theorem~\ref{thm:ollivier} still applies if we assume that $\Gamma_{(1)}$, the quotient by the relations on $G_{i_1}$, satisfies the assumptions of Theorem~\ref{thm:ollivier}. This happens with positive probability if we take $i_1$ large enough, as explained in the discussion preceding Definition 50 in~\cite{Ollivier:RandGpSurvey}

\section{From Poincar\'e inequalities to fixed points}\label{sec:mainthm}

Let $\{G_i\}_{i=1}^\infty$ be an expander family of graphs, with all vertices of degrees
between $3$ and $d$ and $g(G_i) \gtrsim \log|V(G_i)|$.  For later convenience
we assume that the graphs are non-bipartite.  Let $G = \bigsqcup_{i\ge 1} G_i$ be
the disjoint union of the graphs.

Let $\G=\left< S \right>$ be free on the symmetric set of generators
$S$ of size $2k$.  We set $X=\CayG$; a $2k$-regular tree. As in Section~\ref{sec:graphmodel},
for $j\geq 1$ let $\AGj$ denote the space of symmetric
maps from the (directed) edges of $G$ to $S^j$.
Given $\alpha\in\AGj$ let $\Ga$ be the quotient of $\G$ presented by
declaring every word read along a cycle in $G$ to be a relator.
To every $\alpha\in\AGj$ we associate its restrictions $\alpha_k$ to the copy of
$G_k$.

Our model for random groups is obtained by choosing the value of $\alpha$
at each edge independently and uniformly at random.  In Section~\ref{sec:graphmodel} we
reviewed the assumptions on $G_i$ needed so that, with
high probability, the group $\Ga$ is infinite.  We now show that
with probability $1$ the quotient group $\Ga$ has strong fixed-point
properties.

We follow below the lines of \cite{Silberman:RandGpT}, with the natural changes that are required for handling powers $p$ rather than powers $2$, and $p$-uniformly convex metric spaces rather than $\mathrm{CAT}(0)$ spaces.  Moreover, the handling of $j>1$ in~\cite{Silberman:RandGpT} was rather awkward.  Taking advantage of the fact that we are reproducing much of the analysis of~\cite{Silberman:RandGpT}, we give a cleaner argument here for the case $j>1$.

\subsection{Simulating random walks and transferring Poincar\'e inequalities}
\label{sub:transfer-gap}

Let $G$ be a connected finite graph (one of the $G_i$).  We assume
$3\leq\delta(G)\leq\Delta(G)\leq d$ and let $g = g(G)$, $N=|V(G)|$.  We choose
$\alpha\in\AGj$ uniformly at random. In particular, independently
for each edge. Given $u,v\in V(G)$ such that $d_G(u,v)<g/2$, and $x\in X$, let $\beta_{u\to x}(v)$ denote the vertex $x\alpha(\vp)$ of $X$, where $\vp$ is the unique shortest path joining $u$ and $v$ in $G$. Note that, using the notation of Section~\ref{sec:graphmodel},  $\pi_\alpha(\beta_{u\to x}(v))=\alpha_{u\to x}(v)$, where $\pi_\alpha:X\to X_\alpha$ is the natural quotient map.

For every $q\in \mathbb N$, $q<g/2$, we define the random walk $\muqGa$ on the tree $X$ as
follows:
\begin{equation}\label{eq:def-ind-walk}
  \muqGa(x\to\cdot) = \sum_{u\in G} \nu_G(u)
  \big( ({\beta_{u\to x}})_{*} \mu^q_G(u\to\cdot) \big),
\end{equation}
where $\mu_G$ is the standard random walk on $G$ and $\nu_G$ is its stationary measure.
Since $\beta_{u\to\g x}(v) = \g\beta_{u\to x}(v)$, equation~\eqref{eq:def-ind-walk} is a
$\G$-equivariant random walk on $X$.

For any fixed $x,x' \in X$,
$\muqGa(x\to x')$ is a random variable depending on the choice of
$\alpha$. We denote its expectation by $\mubGX^q(x\to
x')\in\WGX$.   It is important to note that while $\mu_G^q$ and
$\muX^q$ are indeed $q$-fold convolutions of the random walks
$\mu_G$ and $\muX$, this is not the case for the other walks
we consider such as $\muqGa$.

The walks $\muqGa(x\to x')$ will now be used to ``simulate''
the walks $\muX^n$ on $X$.  Indeed, with high (asymptotic) probability
the walks $\muqGa(x\to x')$ are close to their expectation values
$\mubGX^q(x\to x')$, and these expectation values can be related
to walks $\muX^n$ for appropriate values of $n$.

Equation \eqref{eq:def-ind-walk} above furnishes the connection
between the averaging notions on $X$ and on $G$.  For  computations,
however, we rewrite it as:
\begin{equation}\label{eq:ind-walk-indicators}
\muqGa(x\to x') =
\sum_{\avp=q} \nu_G(p_0) \mu_G^q(\vp) \mathds{1}\left(x\alpha(\vp)=x'\right),
\end{equation}
where the sum is over all oriented paths $\vp$  of length $q$ in $G$ starting at $p_0$, and
$\mathds{1}(x=y)$ is the characteristic function of the
diagonal of $X\times X$, so that
$\alpha\mapsto\mathds{1}(x\alpha(\vp)=x')$ is an indicator random
variable for the event that $\alpha(\vp)$ equals $x^{-1}x'$ as
elements of $\G$.

We now easily compute the mean walk $\mubGX^q$. We start with the
instructive case $q=1$, where unwinding the definitions of $\nu_G$
and $\mu_G$ gives:
$$ \mu_{G,\alpha}^1(x\to x') = \frac{1}{2\size{E(G)}}
\sum_{\ve\in\vE} \mathds{1}(x\alpha(\ve)=x').$$
Taking expectation we conclude that $\mubGX^1(x\to x')$ equals the
probability that following a random word in $S^j$ will lead us
from $x$ to $x'$, that is $\mu^j_X(x\to x')$.

A similar calculation for $q>1$ gives the following.

\begin{lem}[{generalization of \cite[Lem.\ 2.12]{Silberman:RandGpT}}]\label{lem:avg-ind-walk}
Let $q < g/2$.  We can write $\mubGX^q$ as a convex combination
\begin{equation}\label{eq:avg-ind-walk}
\mubGX^q = \sum_{l=0}^{q} P_G^q(l)\muX^{jl}
\end{equation}
where the weights $P_G^q(l)$ are concentrated on large values of $l$,
in the sense that
\begin{equation}\label{eq:Q}
Q^q_G \eqdef \sum_{l\leq q/6} P_G^q(l) \leq e^{-q/18}. \end{equation}
Also, wherever $\mubGX^q(x\to x')$ is non-zero then it is at least
\begin{equation}\label{eq:def eps}
\epsilon(d,k,j)^q \eqdef \left( \frac{1}{d(2k)^j} \right) ^q.
\end{equation}
\end{lem}
\begin{proof}
Given a path $\vp$ in $G$ of length $q<g/2$, let $\vpp$ be the
shortest path connecting the endpoints of $G$.  Since the ball of radius $q$
in $G$ around the starting vertex $p_0$ of $\vp$ is a tree,
$\vpp$ is unique and can be obtained from $\vp$ by successively cancelling
``backtracks'' (consecutive steps which traverse a single edge in opposite
directions).  This $\vpp$ is a simple path, traversing each of its
edges exactly once.  It follows that the law of the $\Gamma$-valued random
variable $\alpha\mapsto\alpha(\vpp)$ is that of a uniformly chosen element
in $S^{jl}$ where $l=\avpp$.
Moreover, the symmetry of the labelling $\alpha$ shows that the words
$\alpha(\vp)$ and $\alpha(\vpp)$ are equal as elements of the free group
$\Gamma$.  In particular, the expectation of the indicator variable
$\mathds{1}\left(x\alpha(\vp)=x'\right)$ in~\eqref{eq:ind-walk-indicators}
is $\muX^{jl}(x,x')$.  Equation \eqref{eq:avg-ind-walk} now follows, with
$$P_G^q(l) = \sum_{\avp=q, \avpp=l} \nu_G(p_0) \mu_G^q(\vp).$$

Note that $P_G^q(l)$ is precisely the probability that $q$ steps of the
stationary random walk on $G$ travel a distance $l$.  The bound \eqref{eq:Q}
is established in \cite[Lem.\ 2.12]{Silberman:RandGpT}.

For the lower bound on $\mubGX^q(x\to x')$ note first that for any path
$\vp$ in $G$ of length $q$, $\mu_G^q(\vp) \geq d^{-q}$ since every vertex
has degree at most $d$.  Now let $0\leq l \leq q$ and assume that $l,q$ have
the same parity (if either condition fails then $P_G^q(l) = 0$).
Then for any vertex $p_0$ there exists paths $\vp$ of
length $q$ and reduced length $l$ starting at $p_0$.  It follows that
$P_G^q(l) \geq \sum_{p_0} \nu_G(p_0) d^{-q} \geq d^{-q}$ for $l$ as above.

Finally, let $x,x'\in X$ and let their distance be at most $jq$ and
have the same parity as $jq$ (otherwise, for every term
in~\eqref{eq:avg-ind-walk} either $P_G^q(l)$ or $\muX^{jl}(x\to x')$ vanishes).
Then the same argument shows that $\muX^{jq}(x\to x') \geq (2k)^{-jq}$.
Equation \eqref{eq:def eps} now follows from the estimate
$\mubGX^q(x\to x') \geq P_G^q(q) \muX^{jq}(x\to x')$.
\end{proof}

\begin{defn}
We say that $\muGa^\bullet$ \emph{effectively simulates}
$\muX^\bullet$ \emph{up to time} $q_0$ if for every $1\leq q \leq q_0$ and
every $x,x'\in X$ we have:
$$ \muqGa(x\to x') \geq \hf\mubGX^q(x\to x'),$$
and in addition we have for every $x,x'\in X$:
$$ \muGa^1(x\to x') \leq 2\muXj(x\to x').$$
\end{defn}

When the walks on $G$ effectively simulate the walks on $X$ we can transfer
Poincar\'e inequalities from $G$ to $\Ga$:

\begin{prop}\label{pro:transf-gap}
Let $G=(V,E)$ be a finite graph on $N$ vertices, and let $\sigma$ be the
 spectral gap of $G$.  Let $\alpha\in\AGj$ be such that
$\muGa^\bullet$ effectively simulates $\muX^\bullet$ up to a time
$q_0 \gtrsim \log N$. Let $Y$ be a
metric space on which $\Ga$ acts by isometries.
Write $B(X,Y)$ for the space of $\Gamma$-equivariant functions from $X$
to $Y$ where the free group $\Gamma$ acts via its quotient $\Ga$.
Then for every $f\in B(X,Y)$ there exists $m$ comparable to $\log N$ such that
$$\Ep_{\muX^{jm}}(f) \lesssim \left(\Pmodp_Y(\sigma,N)\right)^p \Ep_{\muX^j}(f).$$
\end{prop}
\begin{proof}
By definition of $\muqGa$, we have for $q<g/2$:
\begin{equation}\label{eq:plug beta}
 \Dvar{\muqGa}{f}^p(x) =
\sum_{u\in V}\nu_G(u) \Dvar{\mu_G^q}{f\circ\beta_{u\to
x}}^p(u).
\end{equation}
Note that in~\eqref{eq:plug beta} the composition $f\circ\beta_{u\to
x}$ is well-defined since $\beta_{u\to
x}(v)$ is defined for all $v\in V(G)$ with $d_G(u,v)<g/2$, and $q<g/2$. The same remark applies for the remainder of the computations below, where we treat $\beta_{u\to
x}$ as a function even though it is only a partially defined function.

  Since the action of $\Gamma$ on $Y$ factors via $\Ga$, the function $f$
can also be viewed as an equivariant function on $X_\alpha$.
Fixing $u_0\in V$, we use this to set $f_0 = f\circ\alpha_{u_0\to x}$.
Then for each $u\in V(G)$ we have
$$ \Dvar{\mu_G^q}{f\circ\beta_{u\to x}}^p(u) =
\Dvar{\mu_G^q}{f_0}^p(u),$$ by projecting to $X_\alpha$
and translating by the element $\g\in\G$ which sends
$\alpha_{u_0\to x}(u)$ back to $x$.  It follows that
\begin{equation}\label{eq:simul1}
\Dvar{\muqGa}{f}^p(x) = 2 \Ep_{\mu_G^q}(f_0).
\end{equation}
Applying the Poincar\'e inequality \eqref{eq:poincare-finite}
for maps from $G$ to $Y$ and
using \eqref{eq:simul1} on both sides we have:
\begin{equation}\label{eq:simul2}
\Dvar{\muqGa}{f}^p(x) \lesssim \left(\Pmodp_Y(\sigma(G),N)\right)^p
\Dvar{\muGa}{f}^p(x) .
\end{equation}
If $q$ is small enough then the assumption of effective simulation
allows us to replace the random
walks in \eqref{eq:simul2} by their expectations up to a constant loss.
Applying \lemref{avg-ind-walk} and omitting some (non-negative) terms
in the sum in~\eqref{eq:avg-ind-walk}, we find:
$$ \min_{q\ge l>q/6}\Dvar{\muX^{jl}}{f}^p(x)
        \stackrel{\eqref{eq:Q}}{\le}
   \sum_{q\ge l>q/6} \frac{P_G^q(l)}{1-Q_G^q}\Dvar{\muX^{jl}}{f}^p(x)
        \lesssim
   \left(\Pmodp_Y(\sigma(G),N)\right)^p \Dvar{\muX^j}{f}^p(x).$$
By assumption we can take $q \asymp \log N$, and the proof is complete.
\end{proof}

\begin{prop} (generalization of \cite[Lem.\ 2.13]{Silberman:RandGpT})
Let $G$ be a finite graph with $3\leq\delta(G)\leq\Delta(G)\leq
d$. Let $N=\size{V(G)}$, and assume $g=g(G)\geq C\log N$.  Then there
exists $C' > 0$ depending on $d,k,j,C$ so that
the probability of $\muGa^\bullet$ failing to effectively simulate
$\muX^\bullet$ up to time $C' \log N$ is $o_{d,k,j}(1)$ as $N\to\infty$.
\end{prop}
\begin{proof}
Since $\G$ acts transitively on $X$, our measure-valued random variables
$\muqGa(x\to \cdot)$ are determined by their value at any particular $x\in X$,
which we fix.  For each choice of $\alpha$, the measure $\muqGa(x\to\cdot)$
is supported on the ball $B_X(x,jq)$, so for each $q$ we need to control
$\size{B_X(x,jq)}$ real-valued random variables on $\AGj$.
Let $\muqGa(x\to x')$ be one such random variable.  We give a
bound $\tau_q$ to its Lipschitz constant as a map from $\AGj$
(equipped with the Hamming metric) to $[0,1]$. For this it
suffices to consider a pair of labelings $\alpha,\alpha'$ which
agree everywhere except at $e\in E$. We then have
(sum over paths which traverse $e$ at some point)
$$ \left\vert \muqGa(x\to x') - \mu_{G,\alpha'}^q(x\to x')\right\vert
\leq \sum_{e\in\vp} \nu_G(p_0) \mu_G^q(\vp).$$ There are at most $2q d^{q-1}$
such paths, and each contributes at most $\frac{2d}{3N} 3^{-q}$ to the
right-hand-side since $\nu_G(u) = d(u)/2|E(G)|$.
the vertex degrees allow us to take
$$ \tau_q = \frac{4q}{3N}\left(\frac{d}{3}\right)^q.$$

We would like to rule out $\muqGa(x\to x')$ deviating from its
non-zero mean $\mubGX^q(x\to x')$ by a factor of at least $2$.
It enough to bound the probability of deviation by at least
$\hf\epsilon(d,k,j)^q$, where $\hf\epsilon(d,k,j)$ is as in~\eqref{eq:def eps}).  Azuma's inequality (see, e.g., \cite[Thm.\ 7.2.1]{AS00})  shows that the probability for this is at most:
$$ \exp\left\{-\frac{\epsilon(d,k,j)^{2q}}{8\size{E(G)}\tau_q^2}\right\}.$$
We can choose $C'$ small enough to ensure
$\left(\frac{d}{3\epsilon^2}\right)^{jq} $ is an arbitrary small power
of $N$.  Also $\size{E(G)} \lesssim_d N$, so the probability of deviation
is exponentially small in a positive power of $N$.  The number of random
variables is polynomial in $N$ (it is at most $(2k)^{qj}$ for each $q$)
so we can take the union bound.
A similar analysis shows that probability of some $\muGa^1(x\to
x')$ being too large also decays.
\end{proof}

\subsection{Fixed points}

Returning to $G$ being the union of finite components $G_i$,
we summarize the result of the previous section:

\begin{thm}\label{thm:transfer} Assume that the $\{G_i\}_{i\ge 1}$ are connected
non-bipartite graphs on $N_i$ vertices with vertex degrees in $[3,d]$,
spectral gaps $\sigma(G_i)\geq \sigma>0$ and girths $\gtrsim\log N_i$.
Let $G = \bigsqcup_{i\ge 1} G_i$ and let $\Ga$ be constructed at random from
$\alpha\in\AGj$ with $j$ even.  Then almost surely for every metric space $Y$,
and every action of $\Ga$ on $Y$ by isometries,
there exists arbitrarily large $N_i$ such that for any $f\in B(X_\alpha,Y)$
there exist $m$ comparable to $\log N_i$ such that

$$\Ep_{\muX^{jm}}(f) \lesssim
\left(\Pmodp_Y(\sigma,N_i)\right)^p \Ep_{\muXj}(f).$$
\end{thm}

\begin{thm}\label{thm:techthm} Let $\Ga$ satisfy the conclusion
of Theorem \ref{thm:transfer}.
Let $Y$ be $p$-uniformly convex, and assume that $\Pmodp_Y(\sigma) < \infty$
or, in greater generality, that
$$\lim_{N\to\infty}\left(\frac{\log\log N}{\log N}\right)^{\frac{1}{2p}} \Pmodp_Y(\sigma,N) = 0$$
(in the terminology of Definition~\ref{def:small}, we are assuming that
$Y$ has small Poincar\'e moduli of exponent $p$).
Then every isometric action of $\Ga$ on $Y$ fixes a point.
\end{thm}
\begin{proof}
By Theorems \ref{thm:poincare-to-avg} and \ref{thm:transfer},
there exist arbitrarily large $N$ such that
for any equivariant $f\in B(X_\alpha,Y)$ (identified with its pull-back
to $X$) there is some $m$ comparable to $\log N$ such that:

$$ \Ep_{\muX^{j}}\left(\Ap_{\muX^{jm}} f\right) \lesssim_{p,\uY,j,d}
  \left(Q(N) + \frac{1}{\log N}\right)
    \Ep_{\muX^j} (f),$$
where $Q(N)\to 0$ as $N\to \infty$.  Choosing $N$ large enough,
we see can ensure the existence of $m$ such that
$$\Ep_{\muX^j}\left(\Ap_{\muX^{jm}} f\right) \leq \hf \Ep_{\muX^j} (f).$$

Note that the choice of $N$ was independent of $f$.  Now
Proposition \ref{pro:iter-fp} shows that iterating the averaging
(with $m$ depending on $f$ but bounded by $N$) leads to a sequence
converging to a fixed point (here $\Gamma\bs X$ is a s single point,
so $B(X,Y)$ is non-empty).

In more detail, let $\mu_{X_\alpha}$ denote the standard random walk
on $X_\alpha$.  We have in fact shown the existence of $m$ such that
$$\Ep_{\mu_{X_\alpha}^j}\left(\Ap_{\mu_{X_\alpha}^{jm}} f\right) \leq \hf
\Ep_{\mu_{X_\alpha}^j} (f).$$
In order to apply Proposion \ref{pro:iter-fp} we further need to verify that
a certain graph is connected -- specifically the Cayley graph of $\Ga$ with
respect to the set $S^j$.  Since $j$ is even $S^j$ contains $S^2$
(as sets of elements of $\Ga$), so it is enough to verify that $S^2$ is a set
of generators for $\Ga$.  Indeed, the graphs $G_i$ are non-bipartite and
hence contain odd cycles.  It follows that some relators in $R_\alpha$ have
odd length, so that up to multiplication by a relator, every element of $\Ga$
can be represented by a word in $S$ of even length.
\end{proof}

\begin{rem}
Theorem~\ref{thm:techthm} was formulated for the limiting wild group $\Gamma_\alpha$, i.e., the group corresponding to the infinite graph $G$. Arguing identically for the random group  corresponding to the relations of each $G_i$ separately, we obtain Theorem~\ref{thm:main intro}.
\end{rem}

\bigskip
\noindent{\bf Acknowledgements.} We are very grateful to the
anonymous referee for the careful reading of our manuscript, and for
many helpful suggestions.

\subsection*{Added in proof}
Francois Dahmani pointed out to us that the answer to one of the questions that we asked
in the introduction is known.
Specifically, in \cite[Sec.\ 8]{Kapovich:PolygonReps} it is shown how
to construct a hyperbolic group with the fixed-point property on all
symmetric spaces and buildings associated to linear groups.




\bibliography{Graph_Gps_FP.11-01-14}
\bibliographystyle{abbrv}

\end{document}